%% file: Arxiv_version.tex
\documentclass[opre,sglanonrev]{informs4_arxiv}
\usepackage{eqndefns-left} 

\OneAndAHalfSpacedXII 

\usepackage{tikz}

\usepackage[sort&compress]{natbib}
 \bibpunct[, ]{[}{]}{,}{n}{}{,}%
 \def\bibfont{\small}%
\usepackage{tikz}
\usetikzlibrary{arrows.meta,calc,shadows.blur}
\usepackage[margin=1in]{geometry} 
\usepackage{eqndefns-left} 
\usepackage{lipsum}

\usepackage{bm}
\usepackage{endnotes}
\usepackage{makecell}
\usepackage{xparse}
\usepackage{lipsum}
\usepackage{amsfonts}
\usepackage{graphicx}
\usepackage{hyperref}
\usepackage{amstext}
\usepackage{amsfonts}
\usepackage{amssymb}
\usepackage{xcolor}
\usepackage{mathtools}
\usepackage{bm}
\usepackage{epstopdf}
\usepackage{algorithmic}
\usepackage{algorithm}
\usepackage{adjustbox}
\usepackage{amsmath}
\usepackage{float}
\usepackage{verbatim}
\usepackage{graphicx}
\usepackage{caption}
\usepackage{booktabs}
\usepackage{subcaption}
\usepackage{multirow}
\usepackage{paralist}
\usepackage{enumitem}
\usepackage{amstext}
\usepackage{amsfonts}
\usepackage{amssymb}

\usepackage{wrapfig}

\usepackage[textwidth=2cm]{todonotes}

\usepackage[dvipsnames]{xcolor}
\usepackage{bbm}
\usepackage{thmtools}
\usepackage{thm-restate}
\usepackage{titlesec}
\usepackage{amsopn}
\usepackage{tikz}
\usetikzlibrary{matrix}
\usetikzlibrary{positioning,fit,backgrounds}
\usepackage{wrapfig} 
\usepackage{listings}  
\usepackage{xcolor}   
\usepackage{natbib}
 \bibpunct[, ]{(}{)}{,}{a}{}{,}%
 \def\bibfont{\small}%
\usepackage{setspace}



\definecolor{codeblue}{rgb}{0.0, 0.0, 0.5} 
\definecolor{codekw}{rgb}{0.0, 0.0, 0.5} 

\lstdefinestyle{Pytorch}{
    language         = Python,
    backgroundcolor  = \color{white},
    basicstyle       = \fontsize{8.0pt}{9pt}\selectfont\ttfamily\bfseries,
    columns          = fullflexible,
    breaklines       = true,
    captionpos       = b,
    commentstyle     = \fontsize{4pt}{4pt}\color{codeblue},
    keywordstyle     = \fontsize{4pt}{4pt}\color{codekw},
    morekeywords     = {with,scatter_,norm,sort},
}

\definecolor{Gray}{gray}{0.95}

\newcommand{\R}{\mathbb{R}}

\newcommand{\textrnu}[1]{\textup{\uppercase\expandafter{\romannumeral#1}}}
\newcommand{\textrnl}[1]{\textup{\expandafter{\romannumeral#1}}}

\newcommand{\defeq}{\stackrel{\mathrm{def}}{=}}

\EquationsNumberedThrough    

\TheoremsNumberedThrough     
\ECRepeatTheorems  %


\begin{document}


\RUNAUTHOR{Meng, G\'omez and Mazumder}

\RUNTITLE{Computation of LTS: BnB and Hyperplane Arrangements}

\TITLE{Computation of Least Trimmed Squares: A Branch-and-Bound framework with Hyperplane Arrangement Enhancements}

\ARTICLEAUTHORS{%
\AUTHOR{Xiang Meng}
\AFF{MIT Operations Research Center, \EMAIL{mengx@mit.edu}}

\AUTHOR{Andr\'es G\'omez}
\AFF{USC Daniel J. Epstein Department of Industrial and Systems Engineering, \EMAIL{gomezand@usc.edu}}

\AUTHOR{Rahul Mazumder}
\AFF{MIT Sloan School of Management, Operations Research Center, \EMAIL{rahulmaz@mit.edu}}
} 

\ABSTRACT{%
We study computational aspects of a key problem in robust statistics --- the penalized least trimmed squares (LTS) regression problem, a robust estimator that mitigates the influence of outliers in data by capping residuals with large magnitudes. Although statistically attractive, penalized LTS is NP-hard, and existing mixed-integer optimization (MIO) formulations scale poorly due to weak relaxations and exponential worst-case complexity in the number of observations.
We propose a new MIO formulation that embeds hyperplane arrangement logic into a perspective reformulation, explicitly enforcing structural properties of optimal solutions. We show that, if the number of features is fixed, the resulting branch-and-bound tree is of polynomial size in the sample size.
Moreover, we develop a tailored branch-and-bound algorithm that uses first-order methods with dual bounds to solve node relaxations efficiently. Computational experiments on synthetic and real datasets demonstrate substantial improvements over existing MIO approaches: on synthetic instances with $5000$ samples and $20$ features, our tailored solver reaches a $1\%$ gap in 1 minute while competing approaches fail to do so within one hour. These gains enable exact robust regression at significantly larger sample sizes in low-dimensional settings.

}%





\maketitle

\input{Sections/Introduction}
\input{Sections/Mathemetical_Formulation}

\input{Sections/Optimization}

\input{Sections/Experiments}



%
%
%
\bibliographystyle{informs2014} 
\bibliography{references}

\begin{APPENDICES}
\input{Sections/appendix}

\end{APPENDICES}



\end{document}

%% file: Sections/Introduction.tex
\section{Introduction}

Outliers---informally speaking, observations that deviate substantially from the bulk of the data---are pervasive in modern datasets. They arise from sensor failures, human entry mistakes, transient process instabilities, or simply corruptions in data. 
For example, in regression problems, contaminated observations or outliers can severely affect the quality of statistical estimates compared to estimates obtained from clean data.
Classical least squares estimation is notoriously sensitive to outliers: a single high-leverage observation can arbitrarily distort the regression coefficients, potentially hurting model prediction and/or statistical inference~\citep{huber1981,maronna2006}. Ridge regression (i.e., least squares penalized with a squared $\ell_2$-norm penalty) --- despite its regularization benefits for ill-conditioned problems, inherits this sensitivity since the squared loss grows without bounds with the residual magnitude. To address this limitation, several robust estimators have been studied in the statistics literature~\citep{rousseeuwleroy1987}.

Given a model matrix $\bm{X} \in \mathbb{R}^{n \times p}$ and response vector $\bm{y} \in \mathbb{R}^n$, consider the optimization problem
\begin{equation}\label{eq:LTS}
\min_{\bm{\beta} \in \mathbb{R}^p, \bm{z} \in \{0,1\}^n} \; \frac{1}{2} \sum_{i=1}^n \left( y_i - \bm{x}_i^\top \bm{\beta} \right)^2 (1 - z_i) + \frac{\lambda}{2} \|\bm{\beta}\|_2^2 + \mu \sum_{i=1}^n z_i,
\end{equation}
where $\bm{x}_i^\top$ denotes the $i$-th row of $\bm{X}$, $\bm{\beta}\in\mathbb{R}^p$ denotes the regression coefficients and $r_i:=y_i-\bm{x}_i^\top \bm{\beta}$, $i \in [n]$ are the residuals. We are interested in the traditional setting of robust statistics where the number of observations is larger than the dimension $p$, that is $n > p$. Binary variable $z_i \in \{0,1\}$ indicates whether observation $i$ is an outlier---the weight $(1-z_i)$ serves to exclude large absolute values of the residual $r_{i}$ (for $i \in [n]$) and include all other residuals in the computation of the regression coefficients. The excluded points correspond to outliers, and the remaining points are inliers. The parameter $\lambda \geq 0$ controls ridge-regularization, while $\mu \geq 0$ is a parameter that penalizes the number of outliers. 
Problem \eqref{eq:LTS} is a penalized variant of the Least Trimmed Squares (LTS) estimator introduced by \cite{rousseeuw1984}, which fits a model by minimizing the sum of the smallest squared residuals. The cardinality-constrained LTS places a constraint on the number of trimmed observations, while \eqref{eq:LTS} represents its 
penalized form 
that allows the trimming level to be implicitly determined by the penalty $\mu$. Compared to other robust estimators such as Least Median of Squares (LMS)\citep{rousseeuw1984}, 
LTS has desirable statistical properties --- it enjoys 
$\sqrt{n}$-consistency, and higher asymptotic efficiency~\citep{rousseeuwleroy1987,rousseeuwvandriessen2006}.

The estimator in problem \eqref{eq:LTS} can also be interpreted through the lens of a robust loss functions. Indeed, upon minimizing Problem~\eqref{eq:LTS} over the binary indicators $\bm{z}$, the problem reduces to:
\begin{equation}\label{eq:capped_ls}
\min_{\bm{\beta}\in\mathbb{R}^p}\; \sum_{i=1}^n \phi_{\text{cap}}\!\left(y_i-\bm{x}_i^\top\bm\beta\right) + \frac{\lambda}{2}\|\bm\beta\|_2^2,~~~\text{where}~~
\phi_{\text{cap}}(r)\defeq \min\Bigl\{\frac{1}{2}r^2,\;\mu\Bigr\}.
\end{equation}\begin{wrapfigure}{r}{0.42\textwidth}
    \vspace{-1.2em}
    \centering
    \includegraphics[width=0.4\textwidth]{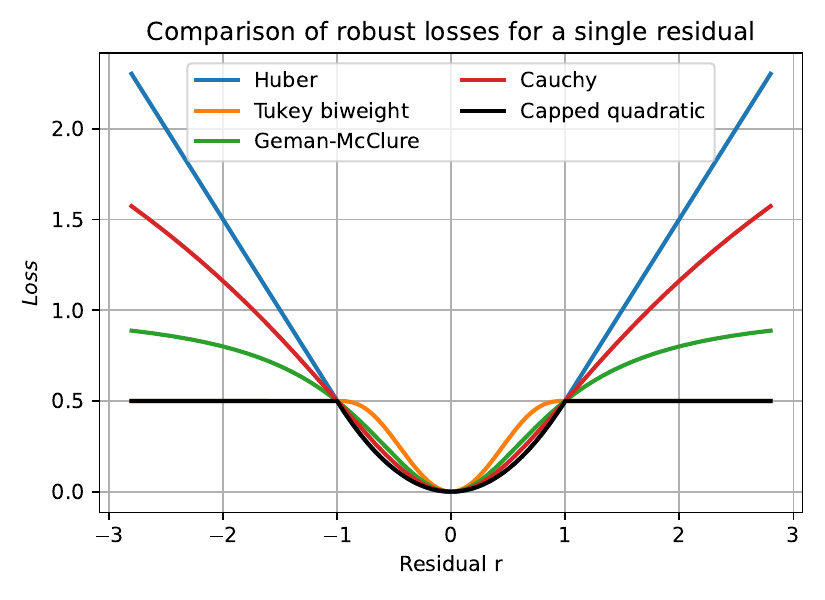}
    \caption{Various robust loss functions.}
    \label{fig:robust_loss}
    \vspace{-1em}
\end{wrapfigure}
The objective has a capped quadratic loss $\phi_{\text{cap}}(r)$ with an additional ridge regularization with penalty parameter $\lambda/2 \geq0$. The loss $r \mapsto \phi_{\text{cap}}(r)$ is quadratic for small (absolute) residuals but saturates at the constant level $\mu$ for large (absolute) residuals, discouraging grossly contaminated observations from dominating the fit. The nonconvex capped quadratic belongs to a broader class of robust losses that temper growth in the tails, as illustrated in Figure~\ref{fig:robust_loss}, and has close ties to the redescending M-estimators~\citep{maronna2006}. A canonical example of a convex loss function is the Huber loss, which transitions from quadratic to linear growth and is a standard baseline in robust regression \citep{huber1964,huber1981}. Well known bounded, nonconvex choices of the loss function include Tukey's biweight \citep{holland1977}, Geman--McClure \citep{chang2012restore}, and Cauchy/Lorentzian losses \citep{motulsky2006}. Relative to these alternatives, the objective in~\eqref{eq:capped_ls} offers a transparent interpretation: observations with sufficiently large residuals contribute a constant effect $\mu$ and are effectively trimmed from the fit.

Despite its attractive properties, solving problem \eqref{eq:LTS} poses a formidable challenge. The problem is NP-hard \citep{bernholt2006}, and early approaches relied on heuristics such as such as the feasible solution algorithm \citep{hawkins1994} and its subsequent refinement \citep{hawkinsolive1999}, or the widely used FAST-LTS
algorithm \citep{rousseeuwvandriessen2006}. 
While efficient, these methods provide no optimality guarantees and can produce solutions that are far away from the optimum~\citep{gomez2025outlier}. In addition, the desirable statistical properties of the LTS estimator 
are predicated on obtaining a global minimizer of \eqref{eq:LTS}, and a suboptimal solution need not inherit any of these guarantees. 

Problem \eqref{eq:LTS} can be formulated as a mixed-integer quadratic program, enabling globally optimal solutions via modern mixed-integer optimization (MIO) solvers \citep{sun2021adaptivecappedls,zioutas2005,gomez2025outlier}. However, these methods have a worst case complexity of $\mathcal{O}(2^n)$. While MIO methods often exhibit practical runtimes that are substantially better than the theoretical worst-case complexities; unfortunately, for the case of \eqref{eq:LTS}, convex relaxations are typically weak and branch-and-bound algorithms require a prohibitive number of nodes to demonstrate optimality. It appears that existing MIO methods for LTS can deliver optimal solutions to problem instances with $n\lesssim 100$ and $p\lesssim 10$ (practical performance is also highly dependent on hyperparameters $\lambda$ and $\mu$ and the dataset itself).

Outside of the MIO literature, some exact combinatorial algorithms have been proposed. \cite{agullo2001} proposed both a probabilistic exchange algorithm and an branch-and-bound algorithm for solving LTS. \cite{hofmann2010}
designed a row adding algorithm that computes exact LTS solutions for a range of coverage values, and \cite{klouda2015} studied an exact algorithm for solving LTS whose computational cost scales as $O(n^{p+1})$. These exact combinatorial methods are generally limited to relatively small or low-dimensional problem instances. 

In a different line of work, exact methods based on topological sweeps for solving LTS have been proposed in the statistics literature. These methods inspired by computational geometry 
are based on an arrangement of hyperplanes, and require solving $\mathcal{O}(n^p)$ least squares problems. While  topological sweep methods can in principle scale to a large number of points provided that $p$ is small, we are unaware of practical implementations apart from the special case of $p=2$ \citep[e.g.,][]{edelsbrunner1990computing,hossjer1995exact}. 
However, as far as we can tell, there are no available practical implementations for the topological sweep method for $p\geq 3$.


\subsection*{Contributions} In this paper we propose new MIO formulations and algorithms for Problem \eqref{eq:LTS}, designed with the common regime of $n\gg p$ in mind. The proposed formulations build on state-of-the-art perspective reformulations for Problem~\eqref{eq:LTS} with strong convex relaxations \citep{gomez2025outlier}, but additionally include additional constraints enforcing hyperplane arrangements logic. The proposed formulation avoids big-M constraints altogether, and branch-and-bound methods based on the proposed formulation terminate after exploring at most $\mathcal{O}(n^{p+1})$ nodes. Unlike topological sweep methods, the practical runtime of branch-and-bound algorithms appear to be substantially better than the worst-case performance. In fact, we propose the first practical implementation of methods based on hyperplane arrangement for Problem~\eqref{eq:LTS} that can handle instances with $p\geq 3$. The proposed formulation leads to at least an order-of-magnitude improvement over alternative state-of-the-art MIO approaches \citep{gomez2025outlier,insolia2022simultaneous}, and the approach of~\cite{bertsimas2014least} for the LMS problem. 

In addition, we also develop a tailored branch-and-bound solver to tackle the proposed formulation. Our algorithm is developed in Python, uses first-order methods to solve the continuous subproblems and requires no commercial software. Our code builds upon and significantly extends the earlier work of~\cite{hazimeh2021sparse} proposing tailored branch-and-bound approaches for a different problem sparse linear models (See Section~\ref{sec:BnBSolver} for differences with this work). The computational performance of our algorithm is competitive with state-of-the-art commercial solver Gurobi to solve the proposed formulation, and appears substantially faster than using other commercial MIO solvers.
On synthetic datasets, our method scales to significantly larger sample sizes: on instances with $n=5000$ samples and $p=20$ features, it reaches a $1\%$ optimality gap in 1 minute, whereas no competing approach does so within a one-hour time limit. On 13 real benchmark regression datasets, our method and Gurobi applied to the proposed formulation achieve the strongest performance, each attaining the best runtime on roughly half of the instances; the remaining baselines are substantially slower or fail to close the gap within the time limit. 


\subsection*{Outline} The rest of the paper is organized as follows. In Section~\ref{sec:review} we review existing methods for \eqref{eq:LTS} (MIO and topological sweep methods as well as heuristics). In Section~\ref{sec:MIOformulation} we describe the proposed MIO formulation, study its strength and prove its worst-case complexity of $\mathcal{O}(n^{p+1})$. In Section~\ref{sec:BnBSolver} we describe the proposed branch-and-bound algorithm based on first-order methods. Finally, in Section~\ref{sec:experiments} we provide computational experiments with synthetic and real data. 

%% file: Sections/Mathemetical_Formulation.tex
\section{Review of existing solution approaches}
\label{sec:review}
We review exact solution approaches for Problem~\eqref{eq:LTS} from the literature. First, we discuss exact MIO approaches that have recently been proposed in the Operations Research community. We then discuss classical exact solution approaches based on hyperplane arrangements from the statistics literature. Finally, we examine approximate methods (also known as heuristics) for finding good solutions, which are currently the most widely used approach in practice.

\subsection{MIO formulations}\label{sec:mip}

At a high level, MIO approaches are based on the branch-and-bound method. The natural approach is to branch on the binary variables $\bm{z}$, that is, branching on whether a data point should be deemed an outlier or not. Earlier approaches from the mathematical optimization community \citep{agullo2001new,giloni2002least} studied formulation \eqref{eq:LTS} directly. However, nonconvexity of the continuous relaxation of Problem~\eqref{eq:LTS} due to the presence of cubic terms in the objective makes the design of efficient solution methods challenging. The first exact MIO approaches were proposed by \citet{zioutas2005deleting} and \cite{zioutas2009quadratic}, which reformulated \eqref{eq:LTS} into mixed-integer quadratically constrained quadratic programs by introducing additional variables representing the loss incurred from fitting each point. More recently, \citet{insolia2022simultaneous} proposed an alternative mixed-integer quadratic optimization reformulation by introducing variables representing  ``corrections" associated with each datapoint, which we review in \S\ref{sec:bigM}. Finally, stronger formulations were proposed in the literature, see \cite{gomez2021outlier} for methods specific to time-series data and \cite{gomez2025outlier} for approaches for the general least trimmed squares problem \eqref{eq:LTS}, which we review in Section~\ref{sec:persp}.

In general, MIO methods have a theoretical worst case complexity of $\mathcal{O}(2^n)$, corresponding to all possible combinations of inliers and outliers. Branch-and-bound methods typically explore substantially fewer branch-and-bound nodes (especially if strong relaxations are used). Nonetheless, both the theoretical complexity and computational experiments reported in the aforementioned papers suggest that the performance of existing MIO methods deteriorates substantially as the sample size $n$ increases. Current MIO approaches have been shown to perform well with real data for $n\approx 100$, but struggle to solve to optimality instances with larger values of $n$.

\subsubsection{Big-$M$ Formulation}\label{sec:bigM}

We first consider a MIO formulation of Problem~\eqref{eq:LTS} similar to the one proposed by \cite{insolia2022simultaneous}, incorporating Big-$M$ constraints. We introduce auxiliary variables $\bm{w} \in \mathbb{R}^n$ representing corrections to the response vector, where $w_i$ captures the residual when observation $i$ is treated as an outlier---that is, $w_i=y_i-\bm{x_i^\top\beta}$. We define a constant $M > 0$ such that any optimal correction satisfies $|w_i^*| \leq M$. Binary variables $z_i, i \in [n]$, indicate whether observation $i$ is an outlier, enforced through the constraints $-M z_i \leq w_i \leq M z_i, \, i \in [n]$. This results in the following Big-$M$ formulation:
\begin{equation}\label{eq:LTS_BigM}
\begin{aligned}
\min_{\bm{\beta}, \bm{w}, \bm{z}} \; & \frac{1}{2} \sum_{i=1}^n \left( y_i - w_i-\bm{x}_i^\top \bm{\beta} \right)^2 + \frac{\lambda}{2} \|\bm{\beta}\|_2^2 + \mu \sum_{i=1}^n z_i \\
\text{s.t.~~} \; & -M z_i \leq w_i \leq M z_i, \quad i \in [n], \\
& \bm{\beta} \in \mathbb{R}^p, \; \bm{w} \in \mathbb{R}^n, \; z_i \in \{0,1\}, \quad i \in [n].
\end{aligned}
\end{equation}
In the above formulation, if $z_i=0$, then the big-M constraints force $w_i=0$; in that case, the loss associated with datapoint $i$ reduces to the standard quadratic loss. On the other hand, if $z_i=1$, then $w_i$ is free to take any value; in particular,  for a fixed $\beta$, setting $w_i=y_i-\bm{x_i^\top\beta}$ is optimal, and the associated loss with datapoint $i$ vanishes.

There are two concerns in using formulation \eqref{eq:LTS_BigM}. Firstly, we need to specify the value of the maximum magnitude $M$, which can be large for arbitrary outliers. This is different from formulation~\eqref{eq:LTS}, which doesn't need such a  specification. If we select $M$ to be sufficiently large so that a solution to Problem~\eqref{eq:LTS_BigM} is also a solution to Problem~\eqref{eq:LTS} \citep[e.g., using arguments similar to][]{bertsimas2016best}, the value of $M$ can be very conservative, possibly causing numerical issues for an algorithm. Additionally, convex relaxations of \eqref{eq:LTS_BigM} are weak since solution $\bm{\beta=0}, \bm{w}=\bm{y}$, $\bm{z}=\epsilon\bm{1}$ are feasible provided $\epsilon\geq \|\bm{y}\|_\infty/M$, with objective values close to $0$ for any reasonable value of $M$. As a consequence, branch-and-bound algorithms struggle to prune nodes and make informed branching decisions, leading to prohibitive solution times. 


\subsubsection{Perspective Formulation}\label{sec:persp}

\citet{gomez2025outlier} propose stronger, big-M free formulations of Problem~\eqref{eq:LTS}, obtained by applying the perspective reformulation \citep{frangioni2006perspective, gunluk2010perspective}.

The key observation is that the objective of Problem~\eqref{eq:LTS_BigM} can be written as
\begin{equation*}
\frac{1}{2} \|\bm{y} - \bm{w} - \bm{X\beta}\|_2^2 + \frac{\lambda}{2} \|\bm{\beta}\|_2^2 = \frac{1}{2}\|\bm{y}\|_2^2 - \bm{y}^\top (\bm{X\beta} - \bm{w}) + \frac{1}{2} \begin{pmatrix} \bm{\beta}^\top & \bm{w}^\top \end{pmatrix} \bm{\Sigma} \begin{pmatrix} \bm{\beta} \\ \bm{w} \end{pmatrix},
\end{equation*}
where $\bm{\Sigma} = \begin{pmatrix} \bm{X}^\top \bm{X} + \lambda \bm{I} & -\bm{X}^\top \\ -\bm{X} & \bm{I} \end{pmatrix}$. Given $\bm{d}\in \mathbb{R}_+^n$, define $$\widetilde{\bm{\Sigma}}_{\bm{d}}= \begin{pmatrix} \bm{X}^\top \bm{X} + \lambda \bm{I} & -\bm{X}^\top \\ -\bm{X} & \bm{I} - 2\text{Diag}(\bm{d}) \end{pmatrix}$$ and note that the following quadratic term can be decomposed as
\begin{equation}\label{eq:quad_decomp}
\frac{1}{2} \begin{pmatrix} \bm{\beta}^\top & \bm{w}^\top \end{pmatrix} \bm{\Sigma} \begin{pmatrix} \bm{\beta} \\ \bm{w} \end{pmatrix} = \frac{1}{2} \begin{pmatrix} \bm{\beta}^\top & \bm{w}^\top \end{pmatrix}  \widetilde{\bm{\Sigma}}_{\bm{d}}\begin{pmatrix} \bm{\beta} \\ \bm{w} \end{pmatrix} + \sum_{i=1}^n d_i w_i^2,
\end{equation}
where the first term remains convex provided that matrix $\widetilde{\bm{\Sigma}}_{\bm{d}}$ is positive semi-definite. 
The perspective reformulation then replaces the separable term $\sum_i d_i w_i^2$ with $\sum_i d_i w_i^2 / z_i$. The latter term equals $d_i w_i^2$ when $z_i = 1$ and enforces $w_i = 0$ when $z_i = 0$, using the convention that $0/0 = 0$ and $x/0 = \infty$ for $x > 0$. We thus obtain the perspective formulation
\begin{equation}\label{eq:LTS_persp}
\begin{aligned}
\min_{\bm{\beta}, \bm{w}, \bm{z}} \; & \frac{1}{2}\|\bm{y}\|_2^2 - \bm{y}^\top (\bm{X\beta} - \bm{w}) + \frac{1}{2} \begin{pmatrix} \bm{\beta}^\top & \bm{w}^\top \end{pmatrix} \widetilde{\bm{\Sigma}}_{\bm{d}} \begin{pmatrix} \bm{\beta} \\ \bm{w} \end{pmatrix} + \sum_{i=1}^n \left( \mu z_i + \frac{d_i w_i^2}{z_i} \right) \\
\text{s.t.~~} \; & -M z_i \leq w_i \leq M z_i, \quad i \in [n], \\
& \bm{\beta} \in \mathbb{R}^p, \; \bm{w} \in \mathbb{R}^n, \; z_i \in \{0,1\}, \quad i \in [n].
\end{aligned}
\end{equation}
Observe that if $\bm{d}>\bm{0}$, then the big-M constraints $-M\bm{z}\leq \bm{w}\leq M\bm{z}$ can be removed from the formulation. The perspective formulation \eqref{eq:LTS_persp} is equivalent to \eqref{eq:LTS_BigM} but yields tighter continuous relaxations, leading to smaller branch-and-bound trees and faster solve times. 

To find a vector $\bm{d}$ that results in the best convex relaxation of Problem~\eqref{eq:LTS_persp}, \citet{gomez2025outlier} propose an algorithm that solves a sequence of positive semi-definite programs (SDPs) with cones of order $p+1$. We now describe an alternative and simpler approach we use in our numerical experiments. First, let $\tilde\lambda\in (0,\lambda)$ be a small number used to guarantee strong convexity and numerical stability. We note that, by the Schur complement, we can guarantee that $\widetilde{\bm{\Sigma}}_{\bm{d}}\succ \bm{0}$ holds by ensuring that $\bm{I} - \bm{X}(\bm{X}^\top \bm{X} + (\lambda-\tilde\lambda) \bm{I})^{-1} \bm{X}^\top - 2\text{Diag}(\bm{d}) \succeq \bm{0}$. Thus, setting \begin{equation}\label{eq:minEigenvalue}d_j=\frac{1}{2}\gamma_{\min}\left(\bm{I} - \bm{X}(\bm{X}^\top \bm{X} + (\lambda-\tilde\lambda) \bm{I})^{-1} \bm{X}^\top\right)\quad \forall j\in [n],\end{equation}
where $\gamma_{\min}(\cdot)$ denotes the minimum eigenvalue of the argument matrix, guarantees that $\widetilde{\bm{\Sigma}}_{\bm{d}}\succ \bm{0}$. While choice \eqref{eq:minEigenvalue} results in a weaker relaxation than using the method proposed by \citet{gomez2025outlier}, this can be computed faster and does not require access to (commercial) conic optimization solvers or specialized algorithms.

\subsection{Solution via hyperplane arrangements}\label{subsect:hyper}

Different from MIO approaches, exact algorithms for \eqref{eq:LTS} were proposed in the statistics literature~\citep{rousseeuw2003robust} based on hyperplane arrangements. The algorithms involved solving $\mathcal{O}(n^p)$ least squares regression problems. Thus, unlike MIO methods, the complexity is polynomial in the number of datapoints $n$ but exponential in the number of features $p$. As far as we can tell, these algorithms are
difficult to implement in practice. We now review methods based on hyperplane arrangements. We note that the method we describe differs from classical algorithms for the least trimmed squares problems, since we describe a method for the penalized version whereas classical approaches tackle cardinality constrained problems with an upper bound on the number of outliers. Nonetheless, the ideas behind the algorithms are similar.

The hyperplane arrangement algorithm relies on two key results: the first (Proposition~\ref{prop:optimal}) concerns structure of optimal solutions of \eqref{eq:LTS}, and the second (Proposition~\ref{prop:hyperplane}) is a classical result from combinatorial geometry \citep{cover1965geometrical,zaslavsky1975facing}.

\begin{proposition}\label{prop:optimal}
Any optimal solution $(\bm{\beta^*},\bm{z^*})$ of Problem~\eqref{eq:LTS} satisfies:
\begin{enumerate}
\item if $z_i^*=0$, then $|y_i-\bm{x_i^\top\beta^*}|\leq \sqrt{2\mu}$;
\item if $z_i^*=1$, then $|y_i-\bm{x_i^\top\beta^*}|\geq \sqrt{2\mu}$.
\end{enumerate}
\end{proposition}

\begin{proposition}\label{prop:spaceDivision}\label{prop:hyperplane}
An arrangement of $n$ hyperplanes in $\R^p$ passing
through the origin divide the space into at most
$2\sum_{i=0}^{p-1}{n\choose i}=\mathcal{O}(n^p)$ regions, where the upper bound is attained if the hyperplanes are in general position.
\end{proposition}

We can use Proposition~\ref{prop:hyperplane} to solve Problem~\eqref{eq:LTS} as follows. Each datapoint $(\bm{x_i},y_i)$ induces three regions in $\mathbb{R}^p$: 
\begin{align*}
B^\leq_i=&\left\{\bm{\beta}\in \mathbb{R}^p: \bm{x_i^\top \beta}\leq y_i-\sqrt{2\mu}
\right\}\\
B^=_i=&\left\{\bm{\beta}\in \mathbb{R}^p: y_i-\sqrt{2\mu}\leq \bm{x_i^\top \beta}\leq y_i+\sqrt{2\mu}
\right\}\\
B^\geq_i=&\left\{\bm{\beta}\in \mathbb{R}^p: \bm{x_i^\top \beta}\geq y_i+\sqrt{2\mu}
\right\}.
\end{align*}
Note that for any solution $\bm{\beta}$ in the band defined by $B^=_i$, setting $z_i=0$ is a better choice than $z_i=1$ in Problem~\eqref{eq:LTS}---that is, point $i$ is an inlier. Similarly, for any $\bm{\beta}\in B^\leq_i\cup B^{\geq}_i$ we have that $z_i=1$ is a better choice in Problem~\eqref{eq:LTS}, thus point $i$ would be categorized as an outlier. The three regions are defined by two hyperplanes, $\bm{x_i^\top \beta}=y_i-\sqrt{2\mu}$ and $\bm{x_i^\top \beta}=y_i+\sqrt{2\mu}$ corresponding to data point $i \in [n]$. Thus, combining all $2n$ hyperplanes creates a division of $\mathbb{R}^p$ into $\mathcal{O}((2n)^p)$ regions, each corresponding to a selection of outliers. Hence an optimal solution to Problem \eqref{eq:LTS} can be obtained by solving a ridge regularized least squares problem for each one of the regions (with a fixed selection of outliers) and choosing the solution that yields the best objective value.

\begin{example}\label{ex:3points}
Consider three points in $\R^2$: $(x_1,y_1)=(-1,8)$, $(x_2,y_2)=(0,0.7)$ and $(x_3,y_3)=(1,1)$. Figure~\ref{fig:planeArrangementExample} (left) shows the three points in the plane as well as the optimal ridge regularized LTS regression line $y=\beta_0+\beta_1x$, obtained by solving \eqref{eq:LTS} with $\lambda=\mu=1$ -- note that the first point is flagged as an outlier in the optimal solution $(\beta_0^*,\beta_1^*)=(0.8,0.1)$. Figure~\ref{fig:planeArrangementExample} (right) shows the associated hyperplane arrangement in the $\bm{\beta}$ space: each point $i$ is associated with a band $B_i^=$, and for any solution $(\beta_0,\beta_1)\in B_i^=$ the best choice is to keep point $i$ as an inlier; point $i$ is discarded as an outlier for choices $(\beta_0,\beta_1)\not\in B_i^=$. Since the three bands do not intersect, we can conclude that $\bm{z}=\bm{0}$ cannot be optimal, thus the solution with all three points as inliers should not be considered. \hfill $\blacksquare$

\begin{figure}[!h]
    \centering
    \includegraphics[width=0.48\linewidth, trim={11cm 6cm 10.5cm 5.8cm},clip]{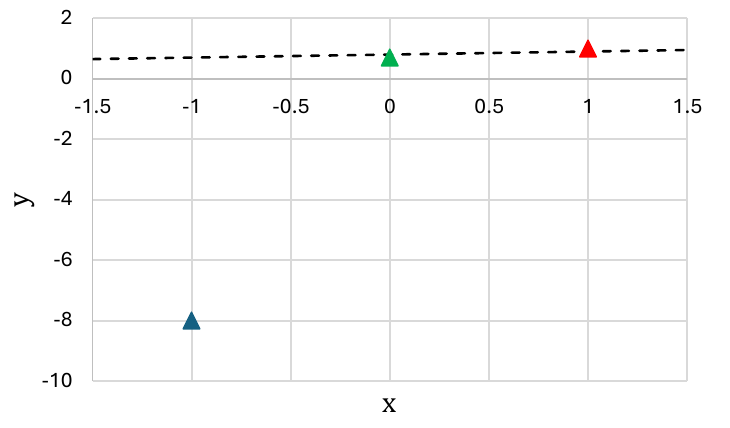}~\includegraphics[width=0.48\linewidth, trim={11cm 6cm 10.5cm 5.8cm},clip]{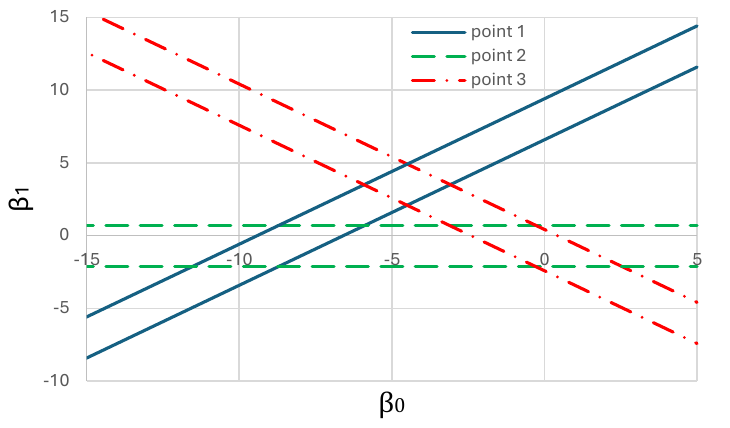}
    \caption{Illustration of Example~\ref{ex:3points}: an ridge regularized LTS regression in the plane (left),  and associated hyperplane arrangement (right). }
    \label{fig:planeArrangementExample}
\end{figure}

\end{example}

Note that for the most part, the algorithmic idea of traversing all regions while solving least squares problems is is computationally expensive and hence may be impractical. Efficient implementations do exist if $p=2$ \citep[e.g.,][]{hossjer1995exact,edelsbrunner1990computing}: using the observation that the intersection of two lines occurs at a single point, one can efficiently enumerate all regions with proper data structures. Unfortunately, we are not aware of practical implementations for $p>2$. In theory, it is possible to do so via linear optimization for example \citep{vcerny2019walks}, but the additional computational cost can be excessive. Moreover, the computational cost of computing $\mathcal{O}(n^p)$ least squares estimators can be prohibitive even if $p=3$ for sufficiently large values of $n$.

\subsection{Heuristics} \label{sec:heuristics}Given the perceived inefficiencies of exact methods, the preferred solution approach for Problem~\eqref{eq:LTS} is via heuristics. The most popular heuristic, called FAST-LTS \citep{rousseeuw2006computing}, is based on alternating minimization and works as follows. Starting from an initial point $\bm{\beta}$, the heuristic alternates between: \textit{(i)} flag the points with largest residuals as outliers (optimization over $\bm{z}$); \textit{(ii)} solve a least squares problem for the new inlier/outlier combination (optimization over $\bm{\beta})$. Variants of this method are common in the literature for problems similar to Problem~\eqref{eq:LTS} \citep{shen2019iterative,shen2019learning}. While alternating minimization heuristics have been observed to work well in high signal-to-noise ratios, and performance guarantees can be established under appropriate conditions \citep{bhatia2015robust}, they have also been shown to lead to suboptimal solutions in challenging instances~\citep{gomez2025outlier}.

\section{Enhanced MIO formulation with hyperplane arrangements}
\label{sec:MIOformulation}
The perspective formulation \eqref{eq:LTS_persp} treats each observation independently: it strengthens the relaxation of each $w_i^2$ term, but does not exploit the relationship between the outlier indicators $\bm{z}$ and the regression coefficients $\bm{\beta}$. Hyperplane arrangements, on the other hand, capture precisely this relationship through the threshold structure of Proposition~\ref{prop:optimal}. We propose a formulation that incorporates this structure into the perspective formulation, by enforcing the logical relationships
\begin{subequations}\label{eq:logicalConstraints}
\begin{align}
z_i=0\implies& \bm{\beta}\in B^{=}_i\label{eq:logicalConstraints_eq}\\
z_i=1\implies& \bm{\beta}\in B^{\leq}_i\cup B^{\geq}_i\label{eq:logicalConstraints_pm}
\end{align}
\end{subequations}
for all datapoints $i\in [n]$.
A direct approach to enforce \eqref{eq:logicalConstraints} is by including additional binary variables $\bm{z^-},\bm{z^+}\in \{0,1\}^n$ such that $\bm{z^-}+\bm{z^+}=\bm{z}$ and in the big-M constraints
\begin{subequations}\label{eq:HA_BigM}
\begin{align}
&y_i+\sqrt{2\mu}-M(1-z_i^+)\leq\bm{x_i^\top \beta}\leq y_i-\sqrt{2\mu}+M(1-z_i^-) \\
&y_i-\sqrt{2\mu}-M(z_i^-+z_i^+)\leq \bm{x_i^\top \beta}\leq y_i+\sqrt{2\mu}+M(z_i^-+z_i^+)
\end{align}
\end{subequations}
for all $i\in [n]$. The additional variables $\bm{z^-},\bm{z^+}$ have the interpretation that ``$z_i^+=1$ if and only if $\bm{\beta}\in B_i^+$" and ``$z_i^-=1$ if and only if $\bm{\beta}\in B_i^-$". Note that constraints \eqref{eq:HA_BigM} are not valid for \eqref{eq:LTS_BigM} or \eqref{eq:LTS_persp} in the usual sense, as they remove feasible integer points. Nonetheless, they do not remove any optimal solutions of the MIO problems and can be used as optimality cuts or constraints.

A limitation of constraints \eqref{eq:HA_BigM} is the presence of the big-M terms. Indeed, the constant $M$ is required to be at least as large as the maximum residual, which in principle can be arbitrarily large. 
We now propose an alternative formulation that imposes the logical considerations \eqref{eq:logicalConstraints} on the perspective formulation of the LTS problem \eqref{eq:LTS_persp} without using big-M constraints:

\begin{subequations}\label{eq:LTS_perspHA}
\begin{align}
\min_{\substack{\bm{\beta}, \bm{w}, \bm{z}\\\bm{w^\pm},\bm{z^\pm}}} \; & \frac{1}{2}\|\bm{y}\|_2^2 - \bm{y}^\top (\bm{X\beta} - \bm{w}) + \frac{1}{2} \begin{pmatrix} \bm{\beta}^\top & \bm{w}^\top \end{pmatrix} \widetilde{\bm{\Sigma}}_{\bm{d}} \begin{pmatrix} \bm{\beta} \\ \bm{w} \end{pmatrix}  \!+\!\sum_{i=1}^n  \left( \mu z_i + \frac{d_i (w_i^+)^2}{z_i^+} +  \frac{d_i (w_i^-)^2}{z_i^-}\right) \nonumber \\
\text{s.t.~~} \; & w_i^- \geq \sqrt{2\mu} \, z_i^-, \quad w_i^+ \geq \sqrt{2\mu} \, z_i^+, \quad i \in [n], \label{eq:LTS_perspHA_linear}\\
& |y_i-\bm{x_i^\top\beta} - w_i| \leq \sqrt{2\mu} (1 - z_i^- - z_i^+), \quad i \in [n], \label{eq:LTS_perspHA_abs}\\
& z_i = z_i^+ + z_i^-, \quad w_i = w_i^+ - w_i^-, \quad i \in [n], \label{eq:LTS_perspHA_equalities}\\
& \bm{\beta} \in \mathbb{R}^p, \; \bm{w} \in \mathbb{R}^n, \bm{z} \in \{0,1\}^n, \bm{w^-},\bm{w^+}\in \mathbb{R}_+^n, \bm{z^-},\bm{z^+}\in \{0,1\}^n.
\end{align}
\end{subequations}
Note that in Problem~\eqref{eq:LTS_perspHA}, the optimization variables $\bm{z}$ and $\bm{w}$ can be projected out and constraints \eqref{eq:LTS_perspHA_equalities} removed, but we keep them for now to facilitate comparisons with Problem~\eqref{eq:LTS_persp}. All constraints in problem~\eqref{eq:LTS_perspHA} are either linear or can be linearized in the case of \eqref{eq:LTS_perspHA_abs}. Intuitively, variables $\bm{w^+}$ and $\bm{w^-}$ represent the positive and negative parts of $\bm{w}$, respectively; variable $z_i^+=1$ ($z_i^-=1)$ if point $i$ is an outlier by overestimating (underestimating) the response variable. The rest of this section is devoted to studying formulation~\eqref{eq:LTS_perspHA}.

\subsection{Formulation Strength}
The proposed formulation \eqref{eq:LTS_perspHA} is at least as strong as the perspective reformulation \eqref{eq:LTS_persp}, and avoids using big-M constraints such as \eqref{eq:HA_BigM}. The perspective reformulation \eqref{eq:LTS_persp} is obtained from the convex hull of the epigraph of the quadratic regularization term and indicator constraints, that is, 
$$Z_{\text{persp}}=\left\{(z,w,t)\in \{0,1\}\times \R^2: t\geq w^2,\; w(1-z)=0 \right\}.$$ Our formulation is based on a richer structure, involving a residual $r_i$ corresponding to a term of the form $y_i-\bm{x_i^\top\beta}$ for some $i\in [n]$. Namely, given $b\in \R_+$, we consider the set 
$$Z_{\text{HA}}(b)=\left\{(z,w,r,t)\in \{0,1\}\times \R^3: t\geq w^2,\; w(1-z)=0,\; |r|(1-z)\leq b(1-z),\; |r|z\geq bz,\; w=rz \right\}.$$
Using decomposition \eqref{eq:quad_decomp}, we can reformulate Problem~\eqref{eq:LTS} as the MIO
\begin{subequations}\label{eq:LTS_IP}
\begin{align}
\min_{\bm{\beta}, \bm{w}, \bm{z},\bm{t}} ~~&~~ \frac{1}{2}\|\bm{y}\|_2^2 - \bm{y}^\top (\bm{X\beta} - \bm{w}) + \frac{1}{2} \begin{pmatrix} \bm{\beta}^\top & \bm{w}^\top \end{pmatrix} \widetilde{\bm{\Sigma}}_{\bm{d}} \begin{pmatrix} \bm{\beta} \\ \bm{w} \end{pmatrix}  \!+\!\sum_{i=1}^n  \left( \mu z_i + d_it_i\right) \\
\text{s.t.} ~~\; & (z_i,w_i,y_i-\bm{x_i^\top \beta},t_i)\in Z_{\text{HA}}(\sqrt{2\mu})\quad \forall i\in [n]\label{eq:LTS_IP_constr}\\
&\bm{\beta}\in \R^p,\;\bm{w}\in \R^n,\;\bm{z}\in \{0,1\}^n,\;\bm{t}\in \R^n.
\end{align}
\end{subequations}
Note that constraints \eqref{eq:LTS_IP_constr} include the epigraph constraint $t_i\geq w_i^2$. In addition:
\begin{enumerate}
\item if $z_i=0$, then the constraint in~\eqref{eq:LTS_IP_constr} 
reduces to $w_i=0$ and $|y_i-\bm{x_i^\top\beta}|\leq \sqrt{2\mu}$, encapsulating precisely relationship \eqref{eq:logicalConstraints_eq}.
\item if $z_i=1$, then the constraint \eqref{eq:LTS_IP_constr} 
reduces to $|y_i-\bm{x_i^\top\beta}|\geq \sqrt{2\mu}$ --corresponding precisely to \eqref{eq:logicalConstraints_pm}-- and $w=y_i-\bm{x_i^\top\beta}$, encapsulating the additional optimality condition that the additional variables $w_i$ are exactly the residuals to ensure that the associated term vanishes.
\end{enumerate}
Naturally, set $Z_{\text{HA}}(b)$ is highly nonconvex due to the presence of binary variables and quadratic constraints. Thus, instead of relying directly on it, our formulation relaxes constraints \eqref{eq:LTS_IP_constr} to
 $(z_i,w_i,y_i-\bm{x_i^\top \beta},t_i)\in \text{cl conv}\left(Z_{\text{HA}}(\sqrt{2\mu})\right),\; \forall i\in [n]$. The description of this convex hull is given in Proposition~\ref{prop:hull}, and the proof is deferred to Appendix~\ref{sec:proofPropHull}.

\begin{proposition}\label{prop:hull}
The closure of the convex hull of $Z_{\text{HA}}(b)$ is given by
$\text{cl conv}\left(Z_{\text{HA}}(b)\right)=\text{proj}_{(z,w,r,t)} \bar Z(b)$ where
\begin{align*}\bar Z(b)=\Big\{(z,w,r,t,w^-,w^+,z^-,z^+):t\geq \frac{(w^+)^2}{z^{+}} + \frac{(w^-)^2}{z^-},\; w^-\geq bz^-,\; w^+\geq bz^+, \\|r-w|\leq b(1-z^--z^+),\;
z^-+z^+=z\leq 1,\; w=w^+-w^-,\; z^-,z^+\geq 0\Big\}.\end{align*}
\end{proposition}

The constraints defining set $\bar Z(b)$ in Proposition~\ref{prop:hull} are precisely \eqref{eq:LTS_perspHA_linear}-\eqref{eq:LTS_perspHA_equalities}. We formalize this property in the following corollary.

\begin{corollary}
    Formulation \eqref{eq:LTS_perspHA} is a correct formulation for Problem~\eqref{eq:LTS}, obtained from Problem~\eqref{eq:LTS_IP} by replacing constraints \eqref{eq:LTS_IP_constr} with $(z_i,w_i,y_i-\bm{x_i^\top \beta},t_i)\in \text{cl conv}\left(Z_{\text{HA}}(\sqrt{2\mu})\right),\; \forall i\in [n]$.
\end{corollary}

So far we have established that formulation \eqref{eq:LTS_perspHA} is strong, in the sense of directly using the convex hull of a region involving a mix of feasibility and optimality conditions. The formulation is indeed the strongest in its class, that is, no stronger formulation can exist unless additional structure is included in the set (e.g., considering sets with multiple datapoints with non-trivial interactions). However, we state in Proposition~\ref{prop:relaxationHA} that the proposed relaxation does not improve the continuous relaxation of \eqref{eq:LTS_persp}. This result is a consequence of the more general Proposition~\ref{prop:perspectiveLoss}, and we defer its proof until then.

\begin{proposition}\label{prop:relaxationHA}
Constraints \eqref{eq:LTS_perspHA_linear}-\eqref{eq:LTS_perspHA_abs} are redundant for the continuous relaxation of Problem~\eqref{eq:LTS_perspHA}. 
\end{proposition}

\begin{corollary}
The optimal solutions and optimal objective values of the convex relaxations of formulations \eqref{eq:LTS_persp} and \eqref{eq:LTS_perspHA} coincide.
\end{corollary}

In light of Proposition~\ref{prop:relaxationHA}, preferring formulation \eqref{eq:LTS_perspHA} over formulation \eqref{eq:LTS_persp} seems counterintuitive: the continuous relaxation is harder to solve due to the presence of additional constraints, and there appear to be no associated gains in terms of relaxation quality. However, we note that the result of Proposition~\ref{prop:relaxationHA} holds only for the relaxation of Problem~\eqref{eq:LTS_perspHA}, and need not hold if any additional constraint is added or objective terms are modified. In particular, the result of Proposition~\ref{prop:relaxationHA} does not hold when additional branching constraints of the form $z_i^\pm\leq 0$ or $z_i^\pm\geq 1$ are introduced. Thus, even if the root relaxations obtained from both formulations coincide, using formulation~\eqref{eq:LTS_perspHA} results in better relaxations at every node other than the root node of the branch-and-bound tree. In fact, as we show in the next section, the inclusion of constraints \eqref{eq:LTS_perspHA_linear}-\eqref{eq:LTS_perspHA_abs} guarantees that branch-and-bound algorithms explore a number of nodes polynomial in the number of datapoints $n$ when the dimension $p$ is fixed.

\subsection{Computational Cost}

Every time a branching constraint $z_i^{\pm}\leq 1$ or $z_i^\pm\geq 1$ is introduced, the feasible values for $\bm{\beta}\in \R^p$ are restricted: the regression coefficients are forced to be on one side of a hyperplane in the hyperplane arrangement induced by $\{\bm{x_i^\top\beta}=y_i\pm \sqrt{2\mu}\}_{i=1}^n$. Every time a branching decision is incompatible with a region of the hyperplane arrangement, then the branch-and-bound node is pruned due to being infeasible for Problem~\eqref{eq:LTS_perspHA}. In contrast, with formulation \eqref{eq:LTS_persp}, no branch-and-bound nodes are infeasible, and pruning only occurs due to bounding or by finding integer solutions. In Proposition~\ref{prop:complexity} we show that, thanks to the possibility of fathoming by infeasibility, branch-and-bound algorithms run in time polynomial in $n$.

\begin{proposition}\label{prop:complexity}
If formulation \eqref{eq:LTS_perspHA} is solved via the branch-and-bound method where:
\begin{itemize}
\item all convex interval relaxations are solved to optimality at each node of the branch-and-bound tree,
\item variable dichotomy is used for branching, that is, branching is performed only via disjunctions ``$z_i^\pm\leq 0$ or $z_i^\pm\geq 1$", and
\item only fractional variables are selected for branching,
\end{itemize}
then the branch and bound tree has at most $\mathcal{O}\left(\min\{4^{n}, n^{p+1}\}\right)$ nodes.
\end{proposition}
\begin{proof}{\textit{Proof}} The computational cost of $\mathcal{O}\left(4^{n}\right)$ is obtained since there are $2n$ binary variables, and a binary tree of depth $2n$ has at most $2\cdot 2^{2n}-1$ nodes.

To prove the complexity of $\mathcal{O}\left(n^{p+1}\right)$, assume the worst-case scenario that the branch-and-bound algorithm never fathoms nodes by bounding, and pruning only occurs at integer or infeasible solutions; these integer or infeasible nodes are thus the only leaves of the trees. Observe that there is a one-to-one correspondence between feasible integer solutions
of Porblem~\eqref{eq:LTS_perspHA} and regions of the hyperplane arrangement induced by $\{\bm{x_i^\top\beta}=y_i\pm \sqrt{2\mu}\}_{i=1}^n$. Thus, from Proposition~\ref{prop:hyperplane} we find that there are at most $\mathcal{O}({n^p})$ leaves corresponding to integer feasible solutions.

We now count the leaves corresponding to infeasible solutions. Consider an arbitrary leaf node that was pruned by infeasibility, which we refer to as the \emph{infeasible leaf}. Since we assumed the convex relaxations are solved to optimality, pruning by infeasibility is identified immediately, as soon as a subproblem is infeasible. In other words, the parent of the infeasible leaf has a feasible continuous relaxation, that is, there exist $\bm{\bar \beta}\in \R^p$ satisfying all constraints \eqref{eq:LTS_perspHA_linear}-\eqref{eq:LTS_perspHA_abs} and additional branching constraints added to reach that node. Therefore, this parent has a feasible descendant; in particular, there is a (unique) descendant obtained by branching at each step according to $\bm{\bar \beta}$ until an integer solution is found. We say that this unique feasible descendant and the original infeasible node are \emph{relatives} to each other. Finally, note that an integer node at depth $q$ of the tree can have at most $q-1$ infeasible relatives, where the upper bound is attained in the case shown in Figure~\ref{fig:unbalanced}. Since the maximum depth is $2n$ and there are at most $\mathcal{O}(n^p)$ integer nodes, it follows that there are at most $\mathcal{O}(2n^{p+1})$ infeasible nodes, at most $\mathcal{O}(n^p+2n^{p+1})$ leaves and twice as many nodes in a binary tree, proving the result.\hfill \Halmos

\begin{figure}[!h]
  \begin{center}
\includegraphics[width=0.25\textwidth, trim={6cm 5cm 16cm 1cm},clip]{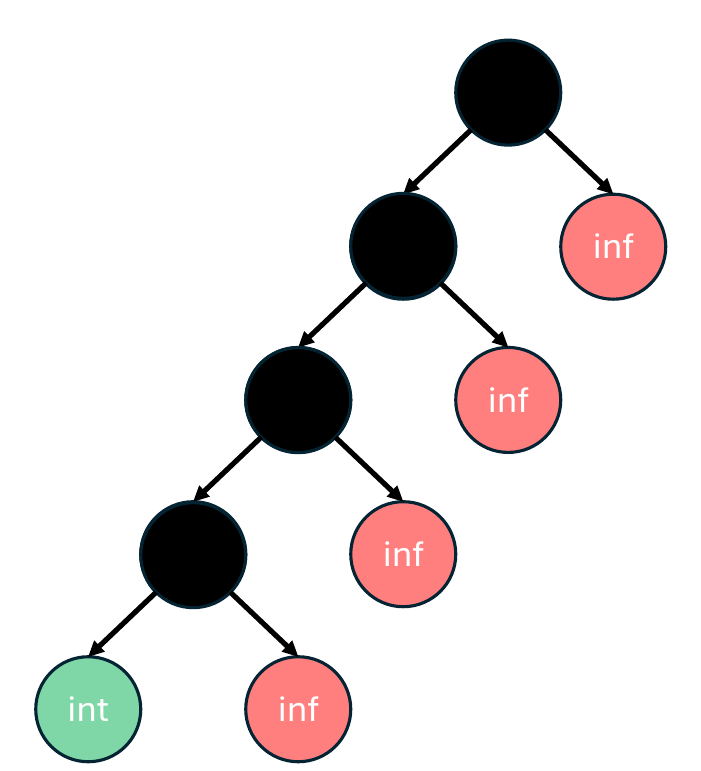}
  \caption{A depth 5 integer nodes with four relatives.}
  \label{fig:unbalanced}
  \end{center}
\end{figure}

\end{proof}

\begin{remark}
Proposition~\ref{prop:complexity} is stated in terms of the proposed formulation \eqref{eq:LTS_perspHA} for simplicity. However, we note that the same proof can be used for any correct formulation including the logic~\eqref{eq:logicalConstraints}. In particular, using the big-M formulation \eqref{eq:LTS_BigM} and big-M constraints \eqref{eq:HA_BigM} would have the same worst-case complexity (provided an adequate value of $M$ can be found). The second and third assumption in Proposition~\ref{prop:complexity} are imposed for simplicity, but the polynomial time complexity holds for most reasonable branching and variable selection strategies. The first assumption---relaxations being solved to optimality---on the other hand is crucial as the proof requires infeasible nodes to be identified immediately.   \hfill $\blacksquare$
\end{remark}

Proposition~\ref{prop:complexity} suggests that formulation \eqref{eq:LTS_perspHA} is particularly effective in low-dimensional settings, even if the number of datapoints (and binary variables) is large. Moreover, Proposition~\ref{prop:complexity} is concerned with a worst-case performance when no bounding occurs at all. In practice, branch-and-bound algorithms perform much better than their worst-case complexity, especially when using strong formulations. Thus, even if computational costs of the order $\mathcal{O}(n^{p+1})$ may appear prohibitively large for $p\approx 3$ and moderately large values of $n$, we show in our computational experiments that the proposed method explores far fewer branch-and-bound nodes in practice and exhibits good computational performance in practice.

Formulation \eqref{eq:LTS_perspHA} can be used with most off-the-shelf MIO solvers by reformulating the perspective terms $\left(w_i^\pm\right)^2/z_i^\pm$ as rotated cone constraints \citep[e.g., see][]{akturk2009strong,gunluk2010perspective} and linearizing the absolute value constraint \eqref{eq:LTS_perspHA_abs}. We found in our computations that this direct approach (with Gurobi as the MIO solver) works quite well provided that $n$ is small (in the low hundreds). However, we also observe that the direct approach can struggle if $n\geq 1000$, even if $p$ is very small. Indeed, formulation \eqref{eq:LTS_perspHA} requires creating $4n$ variables and sets of linear and conic constraints. Solving a large number of convex optimization problems of this form requires a substantial computational effort. In Section~\ref{sec:BnBSolver} we describe a custom branch-and-bound framework to address this issue. Rather than dealing with additional $4n$ variables, we handle the node relaxation problems only in the space of the regression coefficients $\beta \in \mathbb{R}^p$, which is substantially smaller in dimension. This tailored approach avoids the computational overhead incurred by the auxiliary variables introduced in formulation~\eqref{eq:LTS_perspHA}, yielding significant practical efficiency gains, as demonstrated in Section~\ref{sec:experiments}.

\begin{remark}\citet{huchette2019geometric} propose, in a different context, MIO formulations for settings defined by hyperplane arrangements. However, their approach is not comparable with ours, nor can it be directly applied to the setting we consider in a practical manner. At a high level, they propose to consider every region defined by the arrangement and construct an ideal formulation for the complete problem. Such an approach would require enumeration of every region \emph{a priori} (which is computationally demanding for $p\geq 3$), and may lead to a prohibitively large number of regions: for example, in a dataset we consider in our experiments with $n=86$ and $p=7$, the bound from Proposition~\ref{prop:hyperplane} suggests we may need to enumerate $10^9$ regions and then solve a convex problem with billions of decision variables. In contrast, formulation \eqref{eq:LTS_perspHA} does not lead to a strengthening of the continuous relaxation (Proposition~\ref{prop:relaxationHA}), but instead ensures that enumeration of the regions of the arrangement occurs implicitly during branching. Thanks to the ability to bound, the number of regions considered is substantially less than the worst case scenario: for example, in the same instance with $n=86$ and $p=7$, Gurobi terminates after only $15,000$ nodes.
$\blacksquare$
\end{remark}

%% file: Sections/Optimization.tex
\section{A Custom Branch-and-Bound 
Algorithm based on First Order Methods}

\label{sec:BnBSolver}

In this section, 
we present a custom nonlinear branch-and-bound (BnB) framework for Problem~\eqref{eq:LTS_perspHA}, in which the the primal and dual bounds of the node relaxations are obtained via specialized first-order methods. Our work draws inspiration from the work of~\citet{hazimeh2021sparse} for sparse linear models, but has important differences: \citet{hazimeh2021sparse} study ridge regularized least squares problem with a penalty on the number of nonzero regression coefficients. They employ standard binary branching on the binary variables (indicating whether a variable is zero or not), whereas we consider a nonconvex optimization problem expressed directly in terms of $\bm{\beta}\in \R^p$ — having projected out all other variables — and adopt a three-way branching scheme tailored to the structure of our formulation. For additional details on the general BnB framework, we refer the reader to \citet[Chapter~7]{wolsey2020bnb}; details on the specific components of our customized procedure are discussed in Section~\ref{subsect:overview}. Formally, defining function $f:\R^p\to \R$ as

\vspace{-0.6cm}
{\small\begin{align*}f(\bm{\beta})\defeq\min_{\substack{\bm{w},\bm{z}\\\bm{w^\pm},\bm{z}^\pm}}\; & \frac{1}{2}\|\bm{y}\|_2^2 - \bm{y}^\top (\bm{X\beta} - \bm{w}) + \frac{1}{2} \begin{pmatrix} \bm{\beta}^\top & \bm{w}^\top \end{pmatrix} \widetilde{\bm{\Sigma}}_{\bm{d}} \begin{pmatrix} \bm{\beta} \\ \bm{w} \end{pmatrix}  \!+\!\sum_{i=1}^n  \left( \mu z_i + \frac{d_i (w_i^+)^2}{z_i^+} +  \frac{d_i (w_i^-)^2}{z_i^-}\right)\\
\text{s.t.}\;&\eqref{eq:LTS_perspHA_linear},\eqref{eq:LTS_perspHA_abs},\eqref{eq:LTS_perspHA_equalities}\\
&\bm{w}, \bm{w^-},\bm{w^+}\in \mathbb{R}_+^n, \bm{z},\bm{z^-},\bm{z^+}\in \{0,1\}^n,
\end{align*}}
we observe that problem~\eqref{eq:LTS_perspHA} is equivalent to $\min_{\bm{\beta}\in\R^p}f(\bm{\beta})$. Function $\bm\beta\mapsto f(\bm\beta)$ is nonconvex, and our branch-and-bound procedure needs to compute, at each node, a lower bound on $f$ restricted to the subregion defined by the node's branching decisions. In particular, these branching decisions can be described in terms of bounds $\bm{\ell^\pm}, \bm{u^\pm} \in \{0,1\}^n$ with $\bm{\ell^\pm} \leq \bm{u^\pm}$ on the binary variables $\bm{z^\pm}$. Given such bounds, we consider lower bounding functions inspired by the interval relaxation

\vspace{-0.6cm}
{\small
\begin{equation}\label{eq:LTS_perspHA_relax}
\begin{aligned}f_R(\bm{\beta};\bm{\ell^\pm},\bm{u^\pm})\defeq\min_{\substack{\bm{w},\bm{z}\\\bm{w^\pm},\bm{z}^\pm}}\; & \frac{1}{2}\|\bm{y}\|_2^2 - \bm{y}^\top (\bm{X\beta} - \bm{w}) + \frac{1}{2} \begin{pmatrix} \bm{\beta}^\top & \bm{w}^\top \end{pmatrix} \widetilde{\bm{\Sigma}}_{\bm{d}} \begin{pmatrix} \bm{\beta} \\ \bm{w} \end{pmatrix}  \!+\!\sum_{i=1}^n  \left( \mu z_i + \frac{d_i (w_i^+)^2}{z_i^+} +  \frac{d_i (w_i^-)^2}{z_i^-}\right)\\
\text{s.t.}\;&\eqref{eq:LTS_perspHA_linear},\eqref{eq:LTS_perspHA_abs},\eqref{eq:LTS_perspHA_equalities}\\
&\bm{\ell^-}\leq\bm{z^-}\leq \bm{u^-},\bm{\ell^+}\leq\bm{z^+}\leq \bm{u^+}\\
&\bm{w}, \bm{w^-},\bm{w^+}\in \mathbb{R}_+^n, \bm{z},\bm{z^-},\bm{z^+}\in [0,1]^n.
\end{aligned}    
\end{equation}
}

First, in Section \ref{subsect:relaxation}, we describe explicit forms and relaxations of $f_R(\bm{\beta};\bm{\ell^\pm},\bm{u^\pm})$ based on augmented Lagrangians.
Next, in Section \ref{subsect:overview}, we describe the overall branching scheme, pruning conditions and other details for the BnB algorithm. In Section~\ref{sec:dualbound}, we discuss how to solve node relaxations and efficiently compute dual bounds. Finally, in Section~\ref{subsect:implementation} we discuss additional details and computational improvements of the BnB algorithm.

\subsection{Closed form expressions and differentiable lower bounds for $f_R$}\label{subsect:relaxation}

We begin by showing that $f_R$ admits a separable closed-form representation in terms of one-dimensional functions, which forms the basis of our subsequent algorithmic developments.
\begin{proposition}\label{prop:fR_separable}
For any $\bm{\beta} \in \mathbb{R}^p$ and any bounds $\bm{\ell^\pm}, \bm{u^\pm} \in \{0,1\}^n$ with $\bm{\ell^\pm} \leq \bm{u^\pm}$, the relaxation $f_R$ defined in~\eqref{eq:LTS_perspHA_relax} can be written as
\begin{align}\label{eq:perspProjected}
f_R(\bm{\beta}; \bm{\ell^\pm}, \bm{u^\pm}) \;=\; \frac{\lambda}{2}\|\bm{\beta}\|_2^2 + \sum_{i=1}^n \phi\!\left(y_i - \bm{x_i^\top \beta};\, \ell_i^-, \ell_i^+, u_i^-, u_i^+,\, d_i\right),
\end{align}
where, for each $r \in \mathbb{R}$ and $d > 0$, the one-dimensional function $\phi$ is defined as
{\small\begin{subequations}\label{eq:phiGen}
\begin{align}
\phi(r; \ell^-, \ell^+, u^-, u^+, d) \;=\; \min_{\substack{w^-, w^+ \in \mathbb{R}_+ \\ z^-, z^+ \in \mathbb{R}_+}}\;& \tfrac{1}{2}(r - w^+ + w^-)^2 + \mu(z^- + z^+) + \tfrac{d(w^+)^2}{z^+} + \tfrac{d(w^-)^2}{z^-} - d(w^+ - w^-)^2 \\
\text{s.t.}\;& w^- \geq \sqrt{2\mu}\, z^-, \quad w^+ \geq \sqrt{2\mu}\, z^+, \\
& |r - w^+ + w^-| \leq \sqrt{2\mu}\,(1 - z^- - z^+), \label{eq:phiGen_opt} \\
& \ell^- \leq z^- \leq u^-, \quad \ell^+ \leq z^+ \leq u^+, \quad z^- + z^+ \leq 1.
\end{align}
\end{subequations}}
Moreover, provided that $\bm{d}$ satisfies the criteria of Section~\ref{sec:persp}, the function $\bm{\beta} \mapsto f_R(\bm{\beta}; \bm{\ell^\pm}, \bm{u^\pm})$ is convex.
\end{proposition}

The representation~\eqref{eq:perspProjected} follows from the observation that, for fixed $\bm{\beta}$, the inner minimization in~\eqref{eq:LTS_perspHA_relax} is separable across the per-datapoint variable groups $(w_i, z_i, w_i^-, w_i^+, z_i^-, z_i^+)$: constraints~\eqref{eq:LTS_perspHA_linear}–\eqref{eq:LTS_perspHA_equalities} and the bounds on $\bm{z^\pm}$ all decouple across $i$; adding and subtracting $\bm{w}^\top \mathrm{Diag}(\bm{d})\bm{w}$ to the objective and minimizing out $w_i$ in closed form yields the per-datapoint subproblems~\eqref{eq:phiGen} after the substitution $r = y_i - \bm{x_i^\top \beta}$. Convexity of $f_R$ is a direct consequence of the discussion in Section~\ref{sec:persp}.

Proposition~\ref{prop:perspectiveLoss} provides a closed form expression of function $\phi$ when $\bm{\ell^\pm}=\bm{0}$ and $\bm{u^\pm}=\bm{1}$. 
\begin{proposition}\label{prop:perspectiveLoss}If $0<d<1/2$, $\bm{\ell^\pm}=\bm{0}$ and $\bm{u^\pm}=\bm{1}$, then $\phi$ defined in \eqref{eq:phiGen} is given by:
\begin{equation}\label{eq:phi_closed_form0}
\phi(r;0,0,1,1,d) = \begin{cases}
\dfrac{1}{2} r^2, & |r| \leq 2\sqrt{\mu d}\qquad\quad\;\;\; (z^-+z^+=0)\\[8pt]
\dfrac{-d r^2 + 2\sqrt{\mu d} |r| - 2\mu d}{1 - 2d}, & 2\sqrt{\mu d} < |r| < \sqrt{\dfrac{\mu}{d}}, \;\; (0<z^-+z^+<1)\\[8pt]
\mu, & |r| \ge \sqrt{\dfrac{\mu}{d}},\qquad\qquad\; (z^-+z^+=1)
\end{cases}
\end{equation}
where we indicate in parenthesis the optimal values of variables $(z^-,z^+)$. Moreover, the result holds if constraints \eqref{eq:phiGen_opt} are removed from \eqref{eq:phiGen}.
\end{proposition}
Interestingly, the closed form solution reveals new interpretations of the convex relaxations of the perspective formulation \eqref{eq:LTS_persp} and proposed formulation \eqref{eq:LTS_perspHA}: they are equivalent to using a special family of nonconvex loss functions that underestimate the capped quadratic loss $\phi_{\text{cap}}$. 

To simplify the notation, we will use the shorthand:
$$\bar\phi(r;d)\defeq \phi(r;0,0,1,1,d).$$
Observe that the second statement of the proposition, stating that constraints \eqref{eq:phiGen_opt} are redundant, imply Proposition~\ref{prop:relaxationHA}: at the root node, when all lower bounds are zero and upper bounds are one, the hyperplane arrangement constraints do not improve the strength of the perspective relaxation. Figure~\ref{fig:relaxation} illustrates function $r\mapsto \bar\phi(r;d)$ for different values of $d$. We observe that function $r\mapsto \bar\phi(r;d)$ underestimates the capped loss $\phi_{\text{cap}}$ and becomes a better approximation of $\phi_{\text{cap}}$ as $d\to 1/2$, and the two functions coincide for $|r|<2\sqrt{\mu d}$ and $|r|\geq \sqrt{\mu/d}$. The function $r\mapsto \bar\phi(r;d)$ is differentiable: we provide explicit forms of its derivative in Appendix~\ref{app:h_closed_form}. While function $r\mapsto \bar\phi(r;d)$ is nonconvex, we point out that $f_R$ is convex as stated in Proposition \ref{prop:fR_separable} because the presence of the strongly convex term $\frac{\lambda}{2}\|\bm{\beta}\|_2^2$ offsets the nonconvexities introduced by $\bar\phi$.

\begin{figure}[!h]
    \centering
    \includegraphics[width=\linewidth]{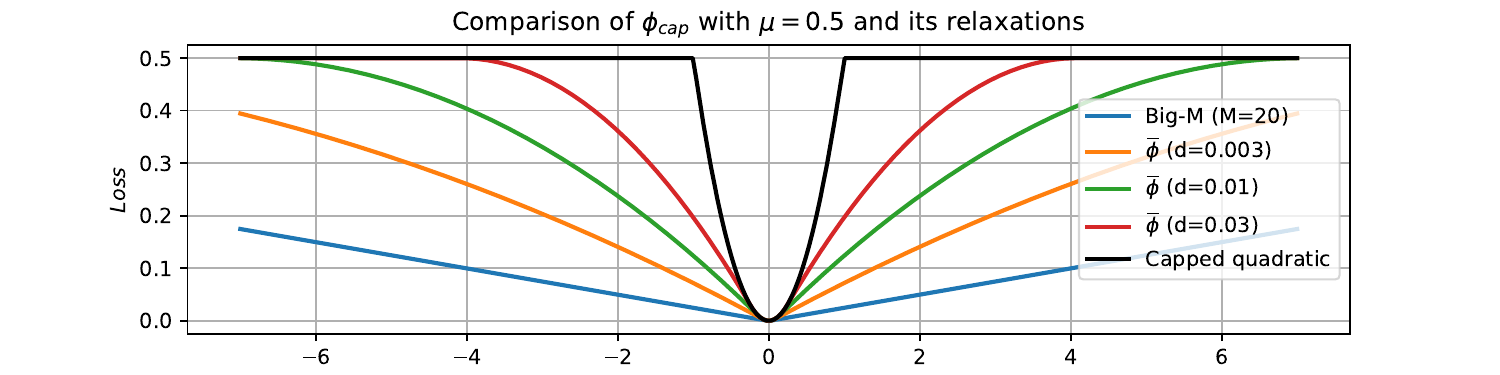}
    \caption{Comparison of loss functions for a single residual $r$: the original least trimmed loss $\phi_{\text{cap}}$ (before relaxation), the perspective relaxation $\bar \phi(r;d)$, and the analogous one-dimensional function obtained by applying the same projection argument to the standard Big-M relaxation of~\eqref{eq:LTS_perspHA} in place of the perspective relaxation. The perspective relaxation provides a tighter convex approximation than the Big-$M$ relaxation.}
    \label{fig:relaxation}
\end{figure}


We conclude this section by deriving explicit closed-form expressions for $\phi$ at the remaining bound configurations.
We also construct differentiable lower bounds for $\phi$:
as the closed-form expressions involve hard indicator constraints on the residual $r$, which effectively turn the minimization of relaxation \eqref{eq:perspProjected} into a constrained optimization problem, preventing direct application of gradient descent. To avoid this, we employ the augmented Lagrangian construction \citep{bertsekas1982constrained}. Concretely, for a constrained problem of the form $\min_{r \in \mathbb{R}} g(r)$ subject to $c_j(r) \leq 0$ for $j = 1, \dots, m$, the corresponding augmented Lagrangian with multipliers $\bm{\nu} \in \mathbb{R}_+^m$ and penalty parameter $\rho \geq 0$ is $\mathcal{L}_\rho(r; \bm{\nu}) = g(r) + \sum_{j=1}^m \nu_j c_j(r) + \frac{\rho}{2}\sum_{j=1}^m [c_j(r)]_+^2$, where $[\,\cdot\,]_+ \defeq \max\{\cdot, 0\}$. Whenever $g$ and the $c_j$'s are differentiable, so is $\mathcal{L}_\rho(\cdot; \bm{\nu})$; moreover, by weak duality, $\min_{r} \mathcal{L}_\rho(r; \bm{\nu}) \leq \min_{r}\{g(r) : c_j(r) \leq 0,\, \forall j\}$ for any $\bm{\nu} \geq \bm{0}$ and $\rho \geq 0$, so that $\mathcal{L}_\rho(\cdot; \bm{\nu})$ serves as a differentiable lower-bounding function for the constrained problem. Applying this construction to each of the constrained forms of $\phi$ that arise during branching yields the following proposition.

\begin{proposition}\label{prop:perspectiveLossSimple} Let $\delta_S(\cdot)$ denote the indicator function of a set $S$, taking value $0$ on $S$ and $+\infty$ elsewhere. For any $d > 0$ and any $\nu, \nu_1, \nu_2, \rho \geq 0$, the function $\phi$ defined in~\eqref{eq:phiGen} admits the following closed-form expressions, each accompanied by a differentiable lower bound (in $r$) obtained from the augmented Lagrangian construction described above:

\noindent $\bullet$ $\phi(r;0,0,0,0,d)=\frac{1}{2}r^2+\delta_{[-\sqrt{2\mu},\sqrt{2\mu}]}(r)$ and a differentiable lower bound is given by 
$$\bar\phi_{0}(r;\nu_1,\nu_2,\rho)=\frac{1}{2}r^2+\nu_1\left(-r-\sqrt{2\mu}\right)+\nu_2\left(r-\sqrt{2\mu}\right)+\frac{\rho}{2}\left(\left[-r-\sqrt{2\mu}\right]_+^2+\left[r-\sqrt{2\mu}\right]_+^2\right).$$

\noindent $\bullet$ $\phi(r;1,0,1,0,d)=\mu+\delta_{[-\infty,-\sqrt{2\mu}]}(r)$ and a differentiable lower bound is given by 
$$\bar\phi_{-}(r;\nu,\rho)=\mu+\nu\left(r+\sqrt{2\mu}\right)+\frac{\rho}{2}\left[r+\sqrt{2\mu}\right]_+^2.$$

\noindent $\bullet$ $\phi(r;0,1,0,1,d)=\mu+\delta_{[\sqrt{2\mu},\infty]}(r)$ and a differentiable lower bound is given by 
$$\bar\phi_{+}(r;\nu,\rho)=\mu+\nu\left(-r+\sqrt{2\mu}\right)+\frac{\rho}{2}\left[-r+\sqrt{2\mu}\right]_+^2.$$
\end{proposition}

The functions $\bar\phi_0, \bar\phi_-, \bar\phi_+$ play a key role in the remainder of the paper: in Section~\ref{subsect:overview}, they are used to construct the augmented Lagrangian of the full node relaxation~\eqref{eq:perspProjected}, and in Section~\ref{sec:dualbound} this Lagrangian is in turn the objective minimized by our augmented Lagrangian method (ALM) for computing dual bounds at each BnB node.

\subsection{Description of the BnB method}\label{subsect:overview}

We now describe the main components of our custom branch-and-bound algorithm.

\noindent \textbf{Root node} The BnB process begins by solving the interval relaxation of problem \eqref{eq:LTS_perspHA} at the root node. From \eqref{eq:perspProjected} we find that the relaxation is given by 
\begin{equation}\label{eq:root}
\min_{\bm{\beta}\in \R^p} \frac{\lambda}{2}\|\bm{\beta}\|_2^2+\sum_{i=1}^n \bar\phi(y_i-\bm{x_i^\top\beta};d_i),\end{equation}
where function $\bar\phi$ is explictly described in Proposition~\ref{prop:perspectiveLoss}. All functions involved are differentiable, thus \eqref{eq:root} can be solved via standard first-order methods. In our approach, we solve it approximately via gradient descent with line search. We describe our specific implementation and the computation of valid dual bounds from approximate solutions in Section~\ref{sec:dualbound}.

\noindent \textbf{Branching variable selection} Once a root (or node) relaxation has been solved, typical branch-and-bound algorithms for mixed-integer optimization would then select a fractional integer variable to branch on. In our case, the binary variables are not defined explicitly, but appear implicitly in the definition of function $\phi$ in \eqref{eq:phiGen}. In particular, function $\bar\phi(\cdot; d)$ underestimates the capped loss $\phi_{\text{cap}}(\cdot)$, and is a strict underestimator if and only if the binary variables appearing in \eqref{eq:phiGen} are fractional. Thus, given an optimal solution $\bm{\bar \beta}\in \R^p$ to Problem~\eqref{eq:root}, the algorithm selects an index $i\in [n]$ such that $\bar\phi(y_i-\bm{x_i^\top\bar \beta};d_i)<\phi_{\text{cap}}(y_i-\bm{x_i^\top\bar \beta})$.

\noindent \textbf{Branching scheme} The definition of $\phi$ in Problem~\eqref{eq:phiGen} involves two binary variables ($z^-$ and $z^+$) and three feasible configurations ($z^-=z^+=0$; $z^-=1$, $z^+=0$; and $z^-=0$, $z^+=1$; note that $z^-=z^+=1$ is infeasible). Thus, instead of using the standard variable dichotomy for branching, we use three-way branching for our custom branch-and-bound. \begin{wraptable}{r}{0.5\textwidth}
    \centering
    \vspace{-0.7cm}
    \begin{tabular}{|c|c|}
        \hline
        Disjunction & Replace $\bar\phi(y_i-\bm{x_i^\top\beta};d_i)$ with... \\ \hline
        $z_i^-=0,z_i^+=0$ & $\bar\phi_{0}(\cdot;\nu_1,\nu_2,\rho)$ \\ 
         $z_i^-=1,z_i^+=0$ & $\bar\phi_{-}(\cdot;\nu,\rho)$ \\ 
          $z_i^-=0,z_i^+=1$ & $\bar\phi_{+}(\cdot;\nu,\rho)$ \\ \hline
    \end{tabular}
    \caption{Subproblem construction}
    \vspace{-5mm}
    \label{tabl:subproblem}
\end{wraptable} After fixing the branching variables, the child node's relaxation objective is obtained from the parent's by replacing $\bar\phi(y_i-\bm{x_i^\top\bar \beta};d_i)$ in the parent node's objective with the appropriate differentiable lower bounds described in Proposition~\ref{prop:perspectiveLossSimple}, as summarized in Table~\ref{tabl:subproblem}.  

\noindent \textbf{Recursion}  The algorithm proceeds recursively, applying the same branching scheme to each newly created node. To describe the relaxation solved at a generic node, let $\mathcal{F}_0, \mathcal{F}_-, \mathcal{F}_+ \subseteq [n]$ denote the disjoint sets of indices recording the branching decisions made along the path from the root to the current node: $i \in \mathcal{F}_0$ if $z_i^- = z_i^+ = 0$ has been fixed, $i \in \mathcal{F}_-$ if $z_i^- = 1$ has been fixed, and $i \in \mathcal{F}_+$ if $z_i^+ = 1$ has been fixed. Following the construction summarized in Table~\ref{tabl:subproblem}, the term $\bar\phi(\cdot; d_i)$ associated with each branched index $i$ is replaced by the corresponding differentiable lower bound from Proposition~\ref{prop:perspectiveLossSimple}, parameterized by dual multipliers $\bm{\nu}$ and a penalty parameter $\rho \geq 0$. This yields the node objective
\begin{align}\label{eq:Lagnode}
L_\rho(\bm{\beta}, \bm{\nu})\defeq &\frac{\lambda}{2}\|\bm{\beta}\|_2^2+\sum_{i\in \mathcal{F}_0} \bar\phi_{0}(y_i-\bm{x_i^\top\beta};(\nu_i)_1,(\nu_i)_2,\rho)+\sum_{i\in \mathcal{F}_-} \bar\phi_{-}(y_i-\bm{x_i^\top\beta};\nu_i,\rho)\notag\\
&+\sum_{i\in \mathcal{F}_+} \bar\phi_{+}(y_i-\bm{x_i^\top\beta};\nu_i,\rho)
+\sum_{i\not\in (\mathcal{F}_0\cup \mathcal{F}_-\cup \mathcal{F}_+)} \bar\phi(y_i-\bm{x_i^\top\beta};d_i).
\end{align}
where $\bm{\nu} \in \mathbb{R}_+^{2|\mathcal{F}_0| + |\mathcal{F}_-| + |\mathcal{F}_+|}$ collects all dual multipliers associated with the branched indices. By Proposition~\ref{prop:perspectiveLossSimple}, each function $\bar\phi_{0}, \bar\phi_{-}, \bar\phi_{+}$ is precisely the augmented Lagrangian of the corresponding constrained function $\phi$. Consequently, $L_\rho(\bm{\beta}, \bm{\nu})$ is the augmented Lagrangian of~\eqref{eq:perspProjected} in which the bounds $(\bm{\ell^\pm}, \bm{u^\pm})$ encode the branching decisions in $\mathcal{F}_0, \mathcal{F}_-, \mathcal{F}_+$. Since the objective in~\eqref{eq:perspProjected} is convex (in $\bm{\beta}$) and all constraints are linear, standard augmented Lagrangian duality \citep{bertsekas1982constrained, rockafellar1976augmented} guarantees that
\begin{equation}\label{eq:node}
    \bar\zeta =  \max_{\bm{\nu} \geq \bm{0}} \min_{\bm{\beta}} L_{\rho}(\bm{\beta}, \bm{\nu}),
\end{equation}
where $\bar\zeta$ denotes the optimal value of the node relaxation \eqref{eq:perspProjected}. Solving the node thus reduces to (approximately) solving this max-min problem; the procedure is described in detail in the next subsection, and produces a primal iterate $\bm{\bar\beta}$ together with a valid lower bound $\bar\zeta'$ on $\bar\zeta$. The node is then processed according to the standard BnB rules: if $\bar\zeta'$ exceeds the current incumbent, the node is pruned by bounding; if $\bar\phi(y_i - \bm{x_i^\top \bar\beta}; d_i) = \phi_{\text{cap}}(y_i - \bm{x_i^\top \bar\beta})$ for all $i \in [n] \setminus (\mathcal{F}_0 \cup \mathcal{F}_- \cup \mathcal{F}_+)$, the node is pruned by integrality and the incumbent is updated (in our implementation, we adopt a best-bound strategy, so this case arises only at the final node explored, if at all); otherwise, an index $i \in [n] \setminus (\mathcal{F}_0 \cup \mathcal{F}_- \cup \mathcal{F}_+)$ with $\bar\phi(y_i - \bm{x_i^\top \bar\beta}; d_i) < \phi_{\text{cap}}(y_i - \bm{x_i^\top \bar\beta})$ is selected for branching, and the recursion continues.


\begin{remark}
Since our BnB implementation does not solve the node relaxation problem to optimality (due to the nature of the first order methods which can be slow to obtain high accuracy solutions), the first condition of Proposition~\ref{prop:complexity} is no longer satisfied. As a consequence, the proposed algorithm is no longer guaranteed to explore at most $\mathcal{O}(n^{p+1})$ nodes. In practice however, when optimally solving the convex interval relaxation at a node becomes expensive, our approach is typically still able to prune the node by bounding, so the total number of nodes explored remains largely unaffected. This is confirmed empirically in Tables~\ref{tab:benchmarks_nodes1}--\ref{tab:benchmarks_nodes2} of Section~\ref{sec:experiments}, which compare the node counts of our solver against those of Gurobi applied directly to Problem~\eqref{eq:LTS_perspHA}.
\end{remark}

\subsection{Node relaxations and dual bounds}\label{sec:dualbound}
At each node of the branch-and-bound algorithm, we compute dual bounds for Problem~\eqref{eq:node}. Observe that for any fixed $\bm{\bar \nu}$
\begin{align}\label{eq:fixedMultiplier}\bar\zeta(\bm{\bar\nu})=\min_{\bm{\beta}}L_\rho(\bm{\beta;\bm{\bar\nu}})
\end{align}
is a lower bound on $\bar\zeta$. Function
$L_\rho(\bm{\beta};\bm{\nu})$ is differentiable as a function of $\bm{\beta}$, and the gradient is
\begin{align}
\nabla_{\bm{\beta}}L_\rho(\bm{\beta}, \bm{\nu})&= \lambda\bm{\beta}-\sum_{i\in \mathcal{F}_0}\bm{x_i} \bar\phi_{0}'(y_i-\bm{x_i^\top\beta};(\nu_i)_1,(\nu_i)_2,\rho)
-\sum_{i\in \mathcal{F}_-}\bm{x_i} \bar\phi_{-}'(y_i-\bm{x_i^\top\beta};\nu_i,\rho)\notag\\
&-\sum_{i\in \mathcal{F}_+}\bm{x_i} \bar\phi_{+}'(y_i-\bm{x_i^\top\beta};\nu_i,\rho)
-\sum_{i\not\in (\mathcal{F}_0\cup \mathcal{F}_-\cup \mathcal{F}_+)}\bm{x_i} \bar\phi'(y_i-\bm{x_i^\top\beta};d_i) \label{eq:grad_fExplicit}
\end{align}
where functions $\bar\phi'$ $\bar\phi_0'$, $\bar\phi_-'$, $\bar\phi_+'$ denote the derivatives, which admit a closed-form expression and are continuous with respect to $\bm\beta$ (see Appendix \ref{app:h_closed_form}). In the proposed approach, we obtain a solution to \eqref{eq:fixedMultiplier} via gradient descent with line search. Note that $\bar\zeta(\bm{\bar\nu})$ is a lower bound only if the associated problem is solved to optimality. However, first order methods can be slow to converge to high-precision solutions. Therefore, we need to consider methods to get reliable lower bounds from an approximate solution, as we discuss next.

We recall that for any $\tilde{\lambda}$-strongly convex function $\bm\beta\mapsto \Phi(\bm\beta)$ where $\Phi(\bm\beta)$ is differentiable, the optimal value satisfies
\begin{equation}\label{eq:grad_gap}
\min_{\bm{\beta}} \Phi(\bm{\beta}) \geq \Phi(\bm{\bar\beta}) - \frac{1}{2\tilde{\lambda}} \|\nabla \Phi(\bm{\bar\beta})\|_2^2
\end{equation}
for any point $\bm{\bar\beta}$. As mentioned in Section~\ref{sec:persp}, the perspective reformulation and relaxations can be ensured to be $\tilde{\lambda}$-strongly convex for any $\tilde\lambda\in (0,\lambda)$, see \eqref{eq:minEigenvalue} (note that larger values of the strong convexity parameter $\tilde\lambda$ come at the expense of weaker convex relaxations and thus less effective pruning of the BnB algorithm). Using this observation, we obtain the following proposition.

\begin{proposition}\label{thm:dual_bound}
Let $({\bm{\bar \beta}},\bm{\bar\nu})$ be any primal-dual pair of Problem \eqref{eq:node}, then the inequality 
\begin{equation}\label{eq:dual_bound}
\bar\zeta(\bm{\bar\nu}) \geq L_\rho({\bm{\bar\beta}}, \bm{\bar\nu}) - \frac{\|\nabla_{\bm{\beta}} L_\rho({\bm{\bar\beta}}, \bm{\bar\nu})\|_2^2}{2\tilde{\lambda}}
\end{equation}
holds true.
\end{proposition}

The proof of Proposition~\ref{thm:dual_bound} is provided in Appendix \ref{app:dual_bound_proof}. Using Proposition~\ref{thm:dual_bound}, we can compute dual bounds for the convex relaxations in the BnB algorithm efficiently from any iterate of a first-order method and its gradient, even when the relaxation is solved approximately.

Building on this result, we solve Problem~\eqref{eq:node} via the augmented Lagrangian method, which alternates between minimizing $L_\rho(\bm{\beta}, \bm{\nu})$ over the primal variables $\bm{\beta}$ with multipliers $\bm{\nu}$ fixed, and updating $\bm{\nu}$. In the primal step, we minimize $L_\rho(\bm{\beta}, \bm{\nu})$ over $\bm{\beta}$ using gradient descent with an Armijo backtracking line search, as detailed in Algorithm~\ref{alg:gd}. The complete augmented Lagrangian procedure for solving~\eqref{eq:node} is summarized in Algorithm~\ref{alg:alm}. Here, Proposition~\ref{thm:dual_bound} is applied at each iteration of both the inner and outer loops to obtain a valid dual bound from the current iterate. As the multipliers $\bm{\nu}$ are progressively updated through the outer loop, the quality of the resulting dual bounds improves.


\begin{algorithm}[!h]
\begin{algorithmic}[1]
\REQUIRE Initial point $\bm{\beta}$, dual values $\bm{\nu}$, initial step size $t_0$, Armijo parameter $\alpha \in (0, 1)$, backtracking factor $\gamma \in (0, 1)$, tolerance $\epsilon$
\STATE $t \leftarrow t_0$
\REPEAT
    \STATE Compute $ \bm{g}=\nabla_{\bm{\beta}} L_\rho(\bm{\beta}, \bm{\nu})$
    \WHILE{$L_\rho(\bm{\beta} - t \bm{g}, \bm{\nu}) > L_\rho(\bm{\beta}, \bm{\nu}) - \alpha t \|\bm{g}\|_2^2$}
        \STATE $t \leftarrow \gamma t$
    \ENDWHILE
    \STATE $\bm{\beta} \leftarrow \bm{\beta} - t \bm{g}$
    \STATE $t \leftarrow t / \gamma$ \hfill $\triangleright$ warm start for next iteration
\STATE Compute dual bound $D$ using Proposition \ref{thm:dual_bound}
\UNTIL{$(L_\rho(\bm{\beta}, \bm{\nu}) - D) / L_\rho(\bm{\beta}, \bm{\nu}) \leq \epsilon$}
\RETURN $\bm{\beta}$
\end{algorithmic}
\caption{Gradient descent with Armijo backtracking for minimizing $L_\rho(\cdot, \bm{\nu})$}
\label{alg:gd}
\end{algorithm}

\begin{algorithm}[!h]
\begin{algorithmic}[1]
\REQUIRE Initial solution $\bm{\beta}^{(0)}$, multipliers $\bm{\nu}^{(0)} \geq \bm{0}$, penalty $\rho > 0$, tolerances $\epsilon_{\text{in}}, \epsilon_{\text{out}}$
\FOR{$k = 0, 1, 2, \ldots$}
    \STATE $\bm{\beta}^{(k+1)} \leftarrow$ minimize $L_\rho(\bm{\beta}, \bm{\nu}^{(k)})$ using Algorithm \ref{alg:gd} with tolerance $\epsilon_{\text{in}}$, starting from $\bm{\beta}^{(k)}$
    \STATE Update dual multipliers: $\nu_j^{(k+1)} = \max\{0, \nu_j^{(k)} + \rho \, c_j(\bm{\beta}^{(k+1)})\}$ for all $j \in \mathcal{C}$
    \STATE \hspace{\algorithmicindent} $\nu_i^{(k+1)} = \max\left\{0, \nu_i^{(k)} + \rho \left(y_i-\bm{x_i^\top}\bm{\beta}^{(k+1)}+\sqrt{2\mu}\right)\right\}$ for all $i \in \mathcal{F}_-$
    \STATE \hspace{\algorithmicindent} $\nu_i^{(k+1)} = \max\left\{0, \nu_i^{(k)} + \rho \left(-y_i+\bm{x_i^\top}\bm{\beta}^{(k+1)}+\sqrt{2\mu}\right)\right\}$ for all $i \in \mathcal{F}_+$
    \STATE \hspace{\algorithmicindent} $(\nu_i)_1^{(k+1)} = \max\left\{0, (\nu_i)_1^{(k)} + \rho \left(-y_i+\bm{x_i^\top}\bm{\beta}^{(k+1)}-\sqrt{2\mu}\right)\right\}$ for all $i \in \mathcal{F}_0$
    \STATE \hspace{\algorithmicindent} $(\nu_i)_2^{(k+1)} = \max\left\{0, (\nu_i)_2^{(k)} + \rho \left(y_i-\bm{x_i^\top}\bm{\beta}^{(k+1)}-\sqrt{2\mu}\right)\right\}$ for all $i \in \mathcal{F}_0$
    \STATE Compute dual bound $D^{(k+1)}$ using Proposition~\ref{thm:dual_bound}
    \IF{$(D^{(k+1)} - D^{(k)}) / D^{(k)} \leq \epsilon_{\text{out}}$} 
        \STATE \textbf{break}
    \ENDIF
\ENDFOR
\RETURN $\bm{\beta}^{(k+1)}$, dual bound $D^{(k+1)}$
\end{algorithmic}
\caption{Augmented Lagrangian Method for solving the relaxation problem}
\label{alg:alm}
\end{algorithm}

\subsection{Implementation Details}\label{subsect:implementation}

We discuss several practical techniques to make our proposed BnB approach computationally efficient. Additionally, we discuss how our work differs from L0BnB \citep{hazimeh2021sparse}, who also use a custom BnB framework with fist order methods to solve the node relaxations for the sparse linear regression which is different from the robust statistics problem we consider here.  

\noindent \textbf{BnB configuration} 
For node selection, we employ a best-bound search strategy \citep{linderoth1999computational}, which selects the open node with the smallest lower bound as the next node to explore. For selecting the branching variable, we branch on the most fractional variable. Note that the integer variables $\bm{z^-},\bm{z^+}$ are not explicitly defined, but can be computed implicitly:
choosing the index $i$ whose $z_i^-+z_i^+$ is closest to $1/2$ is equivalent to choosing
\begin{equation*}
i^\star \in \arg\min_i \left|\,|r_i| - r_0\,\right|,
\quad \text{where}\,\,\,
r_0 := \Bigl(\tfrac12 + d\Bigr)\sqrt{\frac{\mu}{d}},\quad r_i=y_i-\bm{x_i^\top\beta}.
\end{equation*}
Equivalently, since $\phi(r)$ is strictly increasing in $|r|$ when $|r|\in[0, \sqrt{\mu/d}]$,
we find
\begin{equation*}
i^\star \in \arg\min_i \left|\,\phi(r_i) - \phi_0\,\right|,
\qquad
\phi_0 := \phi(r_0)=\frac{\mu\left(\tfrac34-d-d^2\right)}{1-2d}.
\end{equation*}
The BnB procedure terminates when the gap between bounds falls below a predetermined tolerance (1\% in our experiments), or when all nodes have been explored or pruned.

\noindent \textbf{Hyperparameter selection} In our implementation, we set the quadratic penalty term $\rho$ to 100 if $n\ge 300$ and equal to $5$ otherwise. The strong-convexity parameter is set to $\tilde\lambda=0.1\lambda$. We use Armijo type line search with initial step size $t_0=1$, sufficient-decrease parameter $\alpha=10^{-4}$, and backtracking factor $\gamma=0.5$. We set tolerances $\epsilon_{\text{in}}=10^{-4}$ and $\epsilon_{\text{out}}=10^{-4}$ in Algorithm~\ref{alg:alm}. We found the method to be fairly robust to the choice of hyperparameters: $\rho$ was selected from $\{1, 5, 10, 50, 100, 500\}$ on 10 synthetic datasets of various sizes; the remaining hyperparameters were not tuned.

\noindent \textbf{Warm starting} Rather than solving the node relaxation~\eqref{eq:node} from scratch, we warm start Algorithm~\ref{alg:alm} using the parent node's solution. Specifically, $\bm{\beta}^{(0)}$ is initialized to the final primal iterate of the parent node. For each dual multiplier ${\nu}^{(0)}_i$, if the corresponding constraint is also present in the parent node, we inherit its final dual value; otherwise, we set ${\nu}^{(0)}_i = 0$. Since a child node differs from its parent only by fixing a single binary variable $z_i$, the parent's solution typically provides a high-quality starting point, substantially reducing the number of iterations required for convergence in practice.

\noindent \textbf{Selective upper bound computation}
At any node of the BnB tree, we can use heuristics similar to those discussed in Section~\ref{sec:heuristics} based on alternating minimization to compute upper bounds for Problem \eqref{eq:LTS_perspHA}. We provide a precise description of the heuristics in Appendix~\ref{subsect:incumbent}.
As computing incumbent solutions at every node is expensive, we compute these upper bounds only at the root node and nodes whose depth is a multiple of a constant (10 in our experiments). This reduces computational overhead and cost while maintaining effective pruning.

\noindent \textbf{Early pruning.} Since dual bounds via Proposition~\ref{thm:dual_bound} are already computed at every iteration of Algorithms~\ref{alg:gd} and~\ref{alg:alm}, we can monitor them on the fly: as soon as the dual bound exceeds the current incumbent, we terminate the inner solver and prune the node immediately. This step is critical for the efficiency of our BnB procedure, as the original relaxation~\eqref{eq:perspProjected} at a given node may be infeasible, in which case Algorithm~\ref{alg:alm} produces a sequence of dual bounds $D^{(k)}$ diverging to infinity. Without early pruning, the algorithm would expend substantial effort driving $D^{(k)}$ toward infinity on a node that can already be safely discarded. We note that the use of heuristics at the root node guarantees that an incumbent solution is always available for this comparison.


\noindent \textbf{Accelerated gradient computation} We use Numba \citep{lam2015numba}, a just-in-time compiler for Python, to accelerate the gradient evaluation in \eqref{eq:grad_fExplicit}.

\noindent \textbf{Relation to L0BnB}
At a high level, our proposed BnB algorithm extends L0BnB
\citep{hazimeh2021sparse}---a BnB framework originally designed for
$\ell_0$-penalized least squares regression---to the outlier detection problem \eqref{eq:LTS}. Several key modifications are needed to address the distinct challenges posed by \eqref{eq:LTS}. First, we build our BnB on an enhanced MIO formulation \eqref{eq:LTS_perspHA} and derive a new form of the continuous relaxation (Section~\ref{subsect:relaxation}). Second, we develop a three-way branching scheme that efficiently encodes the status of each observation---inlier or outlier, with positive or negative residual (Section~\ref{subsect:overview}). Third, the node relaxations appearing in L0BnB are in the composite form (without constraints) and coordinate descent procedures are used to compute solutions to these problems.
In this work, the relaxation subproblems are more complicated as they have constraints, and we use 
augmented Lagrangian approaches and also introduce a new procedure for extracting valid dual bounds
(Section~\ref{sec:dualbound}).

%% file: Sections/Experiments.tex
\section{Experiments}
\label{sec:experiments}

We evaluate our proposed BnB solver against commercial MIO solvers on both synthetic and real datasets. Our experiments focus on three aspects: computational efficiency compared to existing methods, sensitivity to problem parameters, and the value of exact optimization over heuristic approaches.

\subsection{Experimental Setup}\label{sect:exp_setup}

\noindent \textbf{Solver settings:} We implement our BnB solver in Python with gradient computations accelerated using Numba \citep{lam2015numba}. We compare against Gurobi \citep{gurobi} and Mosek \citep{mosek} applied to the Big-M formulation \eqref{eq:LTS_BigM} (M), the perspective formulation \eqref{eq:LTS_persp} (P), and the strengthened perspective formulation with hyperplane-arrangement \eqref{eq:LTS_perspHA} (H). Accordingly, BnB (P) and BnB (H) denotes our solver on perspective formulation and hyperplane-arrangement formulation, respectively. And Gurobi/Mosek (M, P, H) denotes the commercial solvers on the same formulations. For all solvers, we define the relative optimality gap as $(UB - LB)/UB$, where $UB$ is the incumbent objective value and $LB$ is the best dual bound. We terminate when the gap falls below $1\%$ or when the runtime exceeds 1 hour. All experiments are conducted on a computing cluster, utilizing an AMD EPYC 9474F machine with 10 CPU cores and 40GB RAM.

\noindent \textbf{Synthetic data generation:} Following \cite{hazimeh2020fast}, we generate the data matrix $\bm{X} \in \mathbb{R}^{n \times p}$ by drawing each row from a multivariate Gaussian $\mathcal{N}(\bm{0},\bm{I}_p)$. The true coefficient vector $\bm{\beta}^{\dagger} \in \mathbb{R}^p$ is generated with entries drawn uniformly from $[0, 1]$. We generate clean responses as $\bm{y}_{\text{clean}} = \bm{X} \bm{\beta}^{\dagger} + \bm{\varepsilon}$, where $\bm{\varepsilon} \sim \mathcal{N}(\bm{0}, \sigma^2 \bm{I}_n)$ and $\sigma$ is chosen to achieve a target signal-to-noise ratio $\mathrm{SNR} = \mathrm{Var}(\bm{X} \bm{\beta}^{\dagger}) / \sigma^2$.

To introduce outliers, we contaminate a subset of observations by adding heavy-tailed noise to their responses. Specifically, for each contaminated observation $i$, we set $y_i = y_{\text{clean},i} + \delta \cdot \sigma_y \cdot Z_i$, where $\sigma_y = \mathrm{std}(\bm{y}_{\text{clean}})$, $\delta$ controls the outlier magnitude, and $Z_i$ is drawn from a $t$-distribution with 3 degrees of freedom. Unless otherwise specified, we set  $\mathrm{SNR} = 50$, $\delta = 10$, and introduce 10 outliers in each dataset.

\noindent \textbf{Real datasets.} We evaluate solver performance on 13 benchmark regression datasets: 7 smaller datasets that have been used in prior work on the LTS problem \citep{gomez2025outlier} and 6 larger datasets obtained from the OpenML repository \citep{vanschoren2014openml}. The dimensions of each dataset are summarized in the headers of the corresponding results tables.

\noindent \textbf{Parameter selection:} We set the ridge regularization parameter as $\lambda = \lambda_0 \cdot \mathrm{Mean}(\mathrm{Diag}(\bm{X}^\top \bm{X}))$, where $\lambda_0 = 0.01$ for synthetic datasets and $\lambda_0 = 0.2$ for real datasets, unless otherwise specified. The outlier penalty $\mu$ is chosen so that the model identifies exactly the number of outliers present in the synthetic dataset (or the specified number for real datasets).

\subsection{Comparison on Synthetic Datasets}

\subsubsection{Varying dataset size}
Table \ref{tab:solve_time_avg} compares runtime performance across different numbers of samples $n \in \{1000, 2000, 5000\}$ and features $p \in \{10, 20, 50\}$. The results reveal several key findings. 

\begin{table}[!h]
\centering
\resizebox{\textwidth}{!}{
\setlength{\tabcolsep}{1.8pt}
\small
\begin{tabular}{l |c c c| c c c| c c c}
\toprule
\multirow{2}{*}{Method} & \multicolumn{3}{c|}{\(n=1000\)} & \multicolumn{3}{c|}{\(n=2000\)} & \multicolumn{3}{c}{\(n=5000\)} \\
\cmidrule(lr){2-4}\cmidrule(lr){5-7}\cmidrule(lr){8-10}
 & $p\!=\!10$ & $p\!=\!20$ & $p\!=\!50$ & $p\!=\!10$ & $p\!=\!20$ & $p\!=\!50$ & $p\!=\!10$ & $p\!=\!20$ & $p\!=\!50$ \\
\midrule
*BnB (P) & (10.0\%, 0) & (8.8\%, 1) & (7.0\%, 2) & (14.6\%, 0) & (14.3\%, 0) & (13.8\%, 0) & (17.5\%, 0) & (17.7\%, 0) & (18.4\%, 0) \\
*BnB (H) & 2.5 & 3.6 & 12.9 & 6.3 & 10.0 & 53.4 & 35.7 & 49.2 & 750.6 \\ \midrule
 Gurobi (M) & (33.3\%, 0) & (31.5\%, 0) & (25.6\%, 0) & (53.8\%, 0) & (49.9\%, 0) & (51.0\%, 0) & (65.0\%, 0) & (64.6\%, 0) & (64.2\%, 0) \\
 Gurobi (P) & (5.3\%, 2) & (4.9\%, 3) & (4.5\%, 6) & (13.7\%, 0) & (14.1\%, 0) & (15.3\%, 0) & (19.2\%, 0) & (18.8\%, 0) & (19.4\%, 0) \\
*Gurobi (H) & 358.5 & 414.0 & 766.4 & (1.5\%, 8) & 2301.2 & (6.6\%, 5) & (14.0\%, 0) & (16.1\%, 0) & (19.6\%, 0) \\ \midrule
 Mosek (M) & (32.3\%, 0) & (30.1\%, 0) & (16.9\%, 0) & (47.4\%, 0) & (47.0\%, 0) & (45.4\%, 0) & (58.9\%, 0) & (58.6\%, 0) & (58.8\%, 0) \\
 Mosek (P) & (7.5\%, 0) & (7.5\%, 0) & (4.7\%, 1) & (16.1\%, 0) & (15.4\%, 0) & (15.0\%, 0) & (19.4\%, 0) & (20.0\%, 0) & (20.9\%, 0) \\
*Mosek (H) & (8.6\%, 0) & (1.0\%, 2) & (1.8\%, 0) & (12.5\%, 0) & (11.0\%, 0) & (18.9\%, 0) & (20.1\%, 0) & (20.1\%, 0) & (21.3\%, 0) \\
\bottomrule
\end{tabular}
}
\caption{Runtime performance on synthetic datasets, averaged over 10 random seeds. Methods marked with * are proposed in this work. Methods: (M) Big-M, (P) perspective, (H) hyperplane arrangement formulation. Numeric values indicate solve time in seconds when all 10 instances are solved within 1 hour. Otherwise, we report (gap, count) where gap is the average optimality gap and count is the number of instances solved. }
\label{tab:solve_time_avg}
\end{table}

\paragraph{\textbf{Benchmarking methods}} the Big-M formulation (M) yields the weakest relaxation bounds, with optimality gaps exceeding 30\% for smaller instances and over 60\% for larger ones. The perspective formulation (P) significantly improves upon Big-M, reducing gaps to 5--20\% depending on problem size -- these results are consisting with those reported by \citet{gomez2025outlier}. 

\paragraph{\textbf{Hyperplane arrangement with off-the-shelf solvers}}

The hyperplane arrangement formulation (H) allows for more aggresive pruning of the branch-and-bound tree. When used with Gurobi, problems with $n = 1000$ can now be solved to optimality within minutes, while just using the perspective reformulation results in 5\% gaps. Gurobi is also able to solve more instances and reduce optimality gaps in problems with $n=2,000$, but the benefits seem to disappear in problems with $n=5,000$. The improvements are far less notable with Mosek, and in many cases the use of hyperplane arrangement actually degrades performance. These findings suggest that Gurobi is better equipped to handling problems with heavy constraints. 

\paragraph{\textbf{Hyperplane arrangement with the proposed BnB solver}}
The proposed BnB (H) consistently outperforms all commercial solvers, achieving up to 100$\times$ speedup over Gurobi (H) on comparable instances. This efficiency stems from our reformulation of the relaxation problem in the $\bm{\beta}$-space, which involves only $p$ variables and has low complexity dependence on $n$. For problems with $n = 1000$, the method solves instances in seconds while Gurobi requires several minutes. For the largest instances ($n = 5000$, $p = 50$), BnB (H) completes in approximately 12 minutes, whereas competing methods fail to achieve the 1\% optimality gap within the time limit. Notably, even without the additional hyperplane arrangement constraints, BnB (P) achieves comparable performance to Gurobi (P) for large $n$, confirming the method's scalability with sample size.

\paragraph{\textbf{Additional results with smaller datasets}} We also examine solver performance on smaller datasets. To maintain problem difficulty and clearly distinguish solver performance, we set $\mathrm{SNR} = n/25$ and $\delta \in \{5, 6, 9\}$ for $n \in \{50, 100, 200\}$, respectively. Table \ref{tab:solve_time_small} presents the results.

\begin{table}[ht]
\centering
\small
\begin{tabular}{l c c c c c c c c c}
\toprule
\multirow{2}{*}{Method} & \multicolumn{3}{c}{\(n=50\)} & \multicolumn{3}{c}{\(n=100\)} & \multicolumn{3}{c}{\(n=200\)} \\
\cmidrule(lr){2-4}\cmidrule(lr){5-7}\cmidrule(lr){8-10}
 & $p\!=\!5$ & $p\!=\!10$ & $p\!=\!20$ & $p\!=\!5$ & $p\!=\!10$ & $p\!=\!20$ & $p\!=\!5$ & $p\!=\!10$ & $p\!=\!20$ \\
\midrule
 Gurobi (P) & 2.5 & 2.2 & 4.2 & 21.9 & 26.3 & 19.2 & 447.0 & 87.5 & 97.0 \\
 *Gurobi (H) & 2.0 & 2.8 & 7.1 & 7.5 & 23.3 & 33.3 & 27.2 & 39.8 & 77.9 \\
 *BnB (H) & 2.0 & 3.9 & 34.5 & 10.5 & 12.5 & 11.9 & 9.5 & 9.0 & 50.0 \\
\bottomrule
\end{tabular}
\caption{Runtime performance on small datasets, averaged over 10 random seeds. Methods marked with * are proposed in this work. Numeric values indicate solve time in seconds.}
\label{tab:solve_time_small}
\end{table}

From Table \ref{tab:solve_time_small}, we observe that Gurobi (H) is highly efficient when the dataset scale is very small ($n = 50$), solving instances in 2--7 seconds and outperforming BnB (H). However, Gurobi's runtime increases as $n$ grows. For $n \geq 100$, BnB (H) achieves better performance in most of the settings considered.

\subsubsection{Ablation studies} We investigate how various problem parameters affect solver performance. We fix $n = 2000$ and $p = 20$, adopt the parameter settings described in Section~\ref{sect:exp_setup}, and vary one parameter at a time to conduct the ablation. Table \ref{tab:lambda_outlier} examines the impact of regularization strength $\lambda_0$ (the ridge is defined as $\lambda = \lambda_0 \cdot \mathrm{Mean}(\mathrm{Diag}(\bm{X}^\top \bm{X}))$) and the number of outliers, while Table \ref{tab:snr_delta} analyzes the effect of signal-to-noise ratio and outlier magnitude $\delta$.

\begin{table}[ht]
\centering
\small
\begin{minipage}{0.48\textwidth}
\centering
\begin{tabular}{l c c c}
\toprule
$\lambda_0$ & Gurobi (P) & *Gurobi (H) & *BnB (H) \\
\midrule
 0.0005 & (66.3\%, 0) & (61.7\%, 0) & (32.4\%, 0) \\
 0.001 & (57.9\%, 0) & (53.6\%, 0) & (19.1\%, 0) \\
 0.002 & (46.4\%, 0) & (43.4\%, 0) & (5.2\%, 4) \\
 0.005 & (27.8\%, 0) & (18.2\%, 1) & 34.7 \\
 0.01 & (14.3\%, 0) & (1.2\%, 8) & 10.3 \\
 0.02 & (4.9\%, 0) & 1627.2 & 4.3 \\
 0.05 & 453.3 & 731.6 & 0.9 \\
\bottomrule
\end{tabular}
\end{minipage}
\hfill
\begin{minipage}{0.48\textwidth}
\centering
\begin{tabular}{l c c c}
\toprule
\# outliers & Gurobi (P) & *Gurobi (H) & *BnB (H) \\
\midrule
 5 & (9.9\%, 3) & (1.1\%, 9) & 7.3 \\
 10 & (14.2\%, 0) & (1.0\%, 9) & 9.2 \\
 15 & (15.5\%, 0) & (5.7\%, 5) & (2.8\%, 9) \\
 20 & (18.5\%, 0) & (10.3\%, 4) & (4.7\%, 8) \\
 25 & (16.4\%, 0) & (8.8\%, 5) & (4.2\%, 8) \\
 30 & (17.5\%, 0) & (12.3\%, 4) & (4.4\%, 6) \\
\bottomrule
\end{tabular}
\end{minipage}
\caption{Impact of regularization strength (left) and number of outliers (right). Numeric values indicate solve time in seconds when all 10 instances are solved within 1 hour. Otherwise, we report (gap, count) where gap is the average optimality gap and count is the number of instances solved.}
\label{tab:lambda_outlier}
\end{table}

From Table \ref{tab:lambda_outlier}, we observe that stronger regularization ($\lambda_0 \geq 0.01$) significantly improves tractability for all methods, while weak regularization ($\lambda_0 \leq 0.002$) makes the problem harder due to weaker relaxations. The table also shows that problem difficulty increases with the number of outliers. The BnB solver handles up to 10 outliers efficiently, solving instances in under 10 seconds. As the outlier count increases beyond 15, performance degrades for all methods, though BnB (H) consistently achieves smaller optimality gaps.

\begin{table}[ht]
\centering
\small
\begin{minipage}{0.48\textwidth}
\centering
\begin{tabular}{l c c c}
\toprule
SNR & Gurobi (P) & *Gurobi (H) & *BnB (H) \\
\midrule
 5.0 & (52.7\%, 0) & (52.2\%, 0) & (45.0\%, 0) \\
 10.0 & (45.7\%, 0) & (44.8\%, 0) & (34.8\%, 0) \\
 20.0 & (32.7\%, 0) & (33.6\%, 0) & (13.1\%, 1) \\
 30.0 & (24.1\%, 0) & (18.5\%, 1) & (1.3\%, 8) \\
 40.0 & (18.4\%, 0) & (1.0\%, 9) & 18.9 \\
 50.0 & (14.3\%, 0) & (1.2\%, 8) & 9.4 \\
 100.0 & (2.3\%, 4) & 1268.1 & 3.6 \\
\bottomrule
\end{tabular}
\end{minipage}
\hfill
\begin{minipage}{0.48\textwidth}
\centering
\begin{tabular}{l c c c}
\toprule
$\delta$ & Gurobi (P) & *Gurobi (H) & *BnB (H) \\
\midrule
2.0 & (40.8\%, 0) & (40.7\%, 0) & (36.9\%, 0) \\
 4.0 & (36.9\%, 0) & (36.7\%, 0) & (31.0\%, 0) \\
 6.0 & (29.2\%, 0) & (29.4\%, 0) & (15.5\%, 0) \\
 8.0 & (22.0\%, 0) & (11.3\%, 2) & 474.4 \\
 10.0 & (14.3\%, 0) & (1.7\%, 8) & 9.3 \\
 15.0 & (2.4\%, 5) & 951.2 & 3.3 \\
\bottomrule
\end{tabular}
\end{minipage}
\caption{Impact of signal-to-noise ratio (left) and outlier magnitude (right). Numeric values indicate solve time in seconds when all 10 instances are solved within 1 hour. Otherwise, we report (gap, count) where gap is the average optimality gap and count is the number of instances solved.}
\label{tab:snr_delta}
\end{table}

Table \ref{tab:snr_delta} demonstrates that both higher SNR and larger outlier magnitudes make outliers easier to identify, thereby improving solver performance. When SNR $\leq 20$ or $\delta \leq 6$, distinguishing outliers from inliers becomes difficult, resulting in large optimality gaps for all methods. The proposed solver achieves the best performance across all parameter settings, solving most instances to optimality when SNR $\geq 30$ or $\delta \geq 8$.

\subsection{Evaluation on Real Datasets}

We report solver performance on 13 benchmark regression datasets. For each dataset, we set the outlier penalty $\mu$ so that the model identifies exactly 10 outliers. Table \ref{tab:benchmarks_real1}-\ref{tab:benchmarks_real2} presents the runtime and optimality gap results, and Table \ref{tab:benchmarks_nodes1}-\ref{tab:benchmarks_nodes2} reports the number of BnB nodes explored by each solver. The results are summarized in Figure~\ref{fig:perf_profile_real}.

\begin{table}[ht]
\centering
\small
\resizebox{0.8\textwidth}{!}{
\begin{tabular}{l c c c c c c c}
\toprule
\multirow{2}{*}{Method} & alcohol & education & food & milk & pulpfiber & radar & wagner \\
 & $(44,6)$ & $(50,4)$ & $(150,3)$ & $(86,7)$ & $(62,7)$ & $(1573,4)$ & $(63,6)$ \\
\midrule
 *BnB (P) & (28.6\%) & (14.6\%) & (7.6\%) & (10.9\%) & (29.0\%) & (6.7\%) & (14.7\%) \\
 *BnB (H) & \textbf{3.4} & 118.3 & \textbf{1.0} & 128.9 & \textbf{2.0} & \textbf{23.7} & 56.2 \\
 Gurobi (M) & 11.9 & 754.6 & (18.5\%) & 1564.5 & 16.6 & (23.5\%) & 80.0 \\
 Gurobi (P) & 25.4 & 38.5 & 39.2 & 51.8 & 7.6 & (4.4\%) & 19.3 \\
*Gurobi (H) & 5.4 & \textbf{13.6} & 3.0 & \textbf{28.8} & 10.6 & 660.2 & \textbf{12.9} \\
 Mosek (M) & 127.7 & (8.8\%) & (26.0\%) & (11.2\%) & 129.8 & (23.6\%) & 1234.2 \\
 Mosek (P) & 21.4 & 235.4 & 746.2 & 563.2 & 8.4 & (7.7\%) & 101.8 \\
 *Mosek (H) & (24.8\%) & (--\%) & (--\%) & (12.9\%) & (26.2\%) & (--\%) & (12.9\%) \\
\bottomrule
\end{tabular}
}
\caption{Solver performance on smaller real benchmark datasets with 10 outliers. Methods marked with * are proposed in this work. Dataset dimensions $(n, p)$ are shown below each name. Values in parentheses indicate the optimality gap when the solver fails to reach the 1\% tolerance within the time limit; numeric values indicate solve time in seconds. Bold entries highlight the best result for each dataset. (--\%) indicates solver failure.}
\label{tab:benchmarks_real1}
\end{table}

\begin{table}[ht]
\centering
\small
\resizebox{0.8\textwidth}{!}{
\begin{tabular}{l c c c c c c}
\toprule
\multirow{2}{*}{Method} & pm10 & pollen & no & fri & rmftsa & galaxy  \\
 & $(500,7)$ & $(3848,4)$ & $(500,7)$ & $(1000,50)$ & $(508,10)$ & $(323,4)$ \\
\midrule
 *BnB (P) & (23.8\%) & (19.2\%) & (15.0\%) & (11.4\%) & (26.7\%) & (24.3\%) \\
 *BnB (H) & (15.6\%) & \textbf{(16.7\%)} & (4.6\%) & \textbf{(8.8\%)} & (9.4\%) & (14.7\%) \\
 Gurobi (M) & (31.7\%) & (50.2\%) & (30.2\%) & (37.4\%) & (44.7\%) & (51.3\%) \\
 Gurobi (P) & (20.2\%) & (19.6\%) & (10.4\%) & (10.6\%) & (18.7\%) & (21.7\%) \\
 *Gurobi (H) & \textbf{(14.8\%)} & (19.6\%) & \textbf{985.7} & (10.0\%) & \textbf{1987.8} & \textbf{(4.0\%)} \\
 Mosek (M) & (33.2\%) & (49.1\%) & (32.3\%) & (36.2\%) & (46.6\%) & (52.1\%) \\
 Mosek (P) & (27.4\%) & (19.9\%) & (13.5\%) & (11.3\%) & (21.9\%) & (22.6\%) \\
*Mosek (H) & (--\%) & (19.9\%) & (15.0\%) & (10.9\%) & (--\%) & (23.2\%) \\
\bottomrule
\end{tabular}
}
\caption{Solver performance on larger real benchmark datasets with 10 outliers. Methods marked with * are proposed in this work. Dataset dimensions $(n, p)$ are shown below each name. Values in parentheses indicate the optimality gap when the solver fails to reach the 1\% tolerance within the time limit; numeric values indicate solve time in seconds. Bold entries highlight the best result for each dataset. (--\%) indicates solver failure.}
\label{tab:benchmarks_real2}
\end{table}

\begin{table}[ht]
\centering
\small
\resizebox{0.8\textwidth}{!}{
\begin{tabular}{l c c c c c c c}
\toprule
\multirow{2}{*}{Method} & alcohol & education & food & milk & pulpfiber & radar & wagner \\
 & $(44,6)$ & $(50,4)$ & $(150,3)$ & $(86,7)$ & $(62,7)$ & $(1573,4)$ & $(63,6)$ \\
\midrule
 *BnB (P) & 111960 & 106436 & 88882 & 97331 & 99744 & 105795 & 100338 \\
 *BnB (H) & 1161 & 16069 & 616 & 17530 & 919 & 6253 & 12136 \\
 Gurobi (M) & 80743 & 6666649 & 14265644 & 7271224 & 80587 & 1043403 & 481076 \\
 Gurobi (P) & 84197 & 124485 & 63882 & 109642 & 5560 & 684574 & 40401 \\
 *Gurobi (H) & 4665 & 14080 & 578 & 15205 & 7103 & 9715 & 13031 \\
 Mosek (M) & 131507 & 4816500 & 854332 & 2019163 & 105449 & 80660 & 757731 \\
 Mosek (P) & 12767 & 105370 & 59726 & 90755 & 3022 & 1191 & 31376 \\
 *Mosek (H) & 1129695 & -- & -- & 398930 & 582820 & -- & 695613 \\
 Prop. \ref{prop:complexity} & 2471988 & 41752 & 22652 & $\sim 10^{9}$& $\sim 10^{8}$& $\sim 10^{9}$& $\sim 10^{7}$ \\
\bottomrule
\end{tabular}
}
\caption{Number of BnB nodes explored by each solver on smaller real benchmark datasets. ``--'' indicates solver failure. The last row shows the theoretical upper bound from Proposition \ref{prop:complexity}.}
\label{tab:benchmarks_nodes1}
\end{table}

\begin{table}[ht]
\centering
\small
\resizebox{0.8\textwidth}{!}{
\begin{tabular}{l c c c c c c}
\toprule
\multirow{2}{*}{Method} & pm10 & pollen & no & fri & rmftsa & galaxy \\
 & $(500,7)$ & $(3848,4)$ & $(500,7)$ & $(1000,50)$ & $(508,10)$ & $(323,4)$ \\
\midrule
 *BnB (P) & 92461 & 88564 & 96312 & 91594 & 83874 & 108458 \\
 *BnB (H) & 76871 & 70127 & 81740 & 75705 & 75569 & 76574 \\
 Gurobi (M) & 4748094 & 22293 & 5370154 & 365364 & 4143545 & 4857546 \\
 Gurobi (P) & 1525648 & 75180 & 1764024 & 255223 & 1289427 & 2548249 \\
 *Gurobi (H) & 522409 & 19384 & 170739 & 135830 & 276590 & 1674007 \\
 Mosek (M) & 224224 & 25079 & 260771 & 94819 & 254029 & 372766 \\
 Mosek (P) & 9475 & 47 & 20571 & 4377 & 21400 & 51221 \\
 *Mosek (H) & -- & 47 & 21270 & 3913 & -- & 44744 \\
 Prop. \ref{prop:complexity} & $\sim 10^{14}$ & $\sim 10^{10}$ & $\sim 10^{14}$ & $\sim 10^{104}$ & $\sim 10^{19}$ & $\sim 10^{7}$ \\
\bottomrule
\end{tabular}
}
\caption{Number of BnB nodes explored by each solver on larger real benchmark datasets. ``--'' indicates solver failure. The last row shows the theoretical upper bound from Proposition \ref{prop:complexity}.}
\label{tab:benchmarks_nodes2}
\end{table}

We observe that the hyperplane arrangement formulations (H) with Gurobi or the proposed BnB algorithm consistently outperform the Big-M and perspective formulations across all datasets, demonstrating the value of incorporating threshold constraints (similarly to results with synthetic data, hyperplane arrangement formulations with Mosek perform poorly). Neither BnB (H) or Gurobi (H) consistently outperforms the other: from the performance profile in Figure~\ref{fig:perf_profile_real}, BnB (H) seems to be able to prove optimality faster in easier instances (less than 10 seconds), and Gurobi (H) seems to be more effective in instances requiring more than 1,000 seconds.

Tables \ref{tab:benchmarks_nodes1}--\ref{tab:benchmarks_nodes2} reveal that the hyperplane arrangement formulation (H) requires fewer nodes to solve to optimality compared to the perspective formulation (P), but each node takes longer to solve. This is evident from instances where all solvers reach the time limit: the perspective formulation explores more nodes than the strengthened formulation. Comparing Gurobi (H) to Gurobi (P), the hyperplane arrangement formulation explores only 30--60\% as many nodes. In contrast, BnB (H) explores approximately 80\% of the nodes that BnB (P) does. In other words, while the time required to solve the continuous relaxations increases with the hyperplane arrangement considerations, the proposed BnB method is comparatively less impacted, showcasing the benefit of using first order methods to solve the continuous relaxations.

We observe that, in general, Gurobi (H) requires fewer branch-and-bound nodes than BnB (H) to prove optimality, although: \textit{(i)} there are notable exceptions, see for example results with the ``pulpfiber" dataset;  \textit{(ii)} the differences in number of nodes are typically small. In other words, we observe that the proposed BnB method does not seem to be substantially hampered from not solving the constrained problems to optimality.
The last row of each table shows the theoretical upper bound on the number of nodes required for the strengthened formulation from Proposition \ref{prop:complexity}. In practice, both BnB (H) and Gurobi (H) require only a small fraction of these nodes, showing that the methods perform much better in practice than what the worst-case theoretical bound would suggest. 

\begin{figure}[t]
\centering
\includegraphics[width=0.88\textwidth]{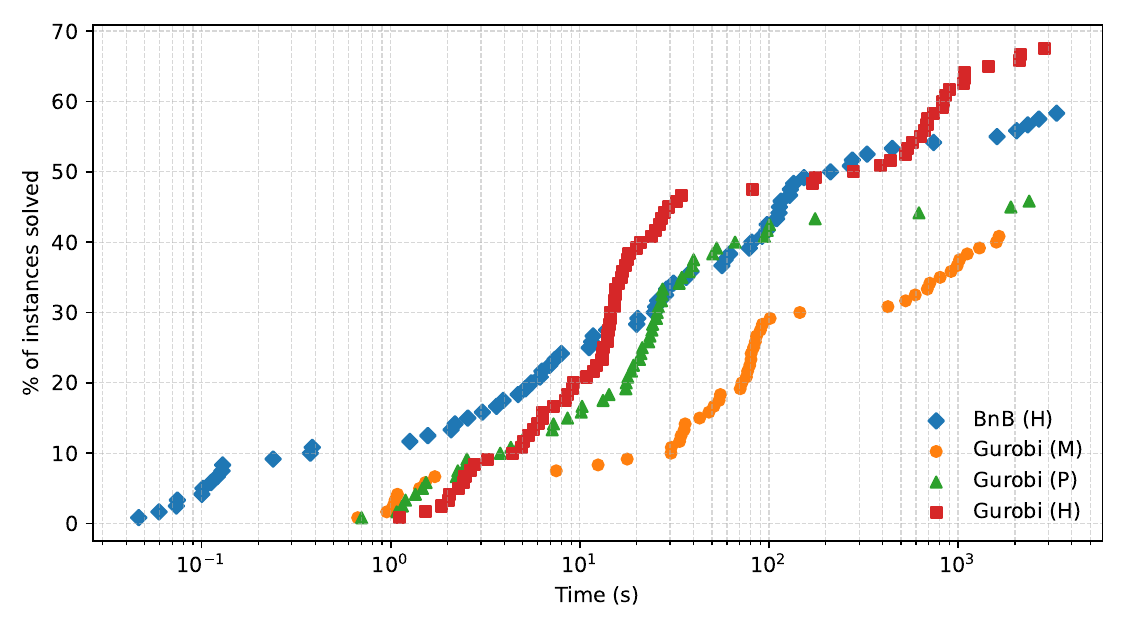}
\caption{Performance profile on the real datasets. For each time $t$, the curve reports the fraction of instances solved within time $t$ (to the 1\% optimality-gap tolerance).}
\label{fig:perf_profile_real}
\end{figure}

Overall, as summarized in Figure~\ref{fig:perf_profile_real}, BnB (H) attains the highest solve rate at short-to-moderate runtimes, while Gurobi (H) remains competitive and becomes advantageous on a subset of harder instances. Both substantially outperform the remaining formulations, confirming that using hyperplane arrangement (H) is the dominant driver of performance on these benchmarks.

\subsection{Value of Optimal Solutions}
We assess the practical benefit of computing optimal solutions compared to heuristic approaches. For each real dataset, we vary the number of outliers from 5 to 20 and compare the objective values obtained by the popularly used alternating minimization heuristic (Algorithm \ref{alg:incumbent}) against the solutions computed by the BnB solver. Figure \ref{fig:quality} presents the relative suboptimality gap $(f_{\text{heuristic}} - f_{\text{optimal}}) / f_{\text{optimal}}$ as a box plot.

\begin{figure}[!h]
    \centering
    \includegraphics[width=\linewidth]{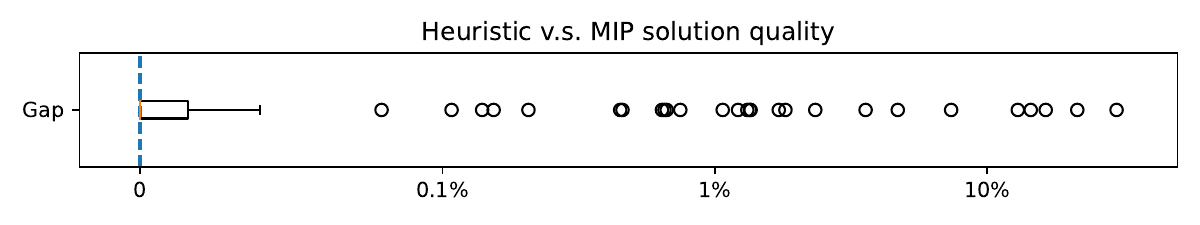}
    \caption{Box plot of relative suboptimality of heuristic solutions (Algorithm \ref{alg:incumbent}) compared to optimal solutions from BnB (H) on 13 real datasets, with outlier counts from 5 to 20.}
    \label{fig:quality}
\end{figure}

The figure demonstrates the value of exact optimization. By construction, the optimal method always achieves an objective value at least as good as the heuristic (since the heuristic is called at the root node). In practice, the optimal method achieves a strictly better objective in more than 50\% of cases. Furthermore, the gap exceeds 1\% in more than 10\% of cases, with the largest gap reaching approximately 30\%. These findings indicate that while heuristics provide reasonable solutions in many cases, exact methods can yield substantially better solutions for a significant fraction of problem instances.

%% file: Sections/appendix.tex
\newpage

\section{Additional Technical Details}\label{sect:app-technical}

\subsection{Proof of Proposition~\ref{prop:hull}}\label{sec:proofPropHull}

In this section we derive the convex hull of set $Z_{\text{HA}}(b)$, which we repeat for convenience. 
$$Z_{\text{HA}}=\left\{(z,w,r,t)\in \{0,1\}\times \R^3: t\geq w^2,\; w(1-z)=0,\; |r|(1-z)\leq b(1-z),\; |r|z\geq bz,\; w=rz \right\};$$
we omit the explicit dependence on parameter $b\in \R_+$ for convenience, as it is fixed for the purposes of this section.

To prove the result, we use disjunctive programming \citep{ceria1999convex}. Indeed, set $Z_{\text{HA}}=Z_-\cup Z_+\cup Z_=$ where
\begin{align*}
Z_-&\defeq\left\{(z,w,r,t)\in \R^4: t\geq w^2, z=1,r\leq -b, w=r \right\}\\
Z_+&\defeq\left\{(z,w,r,t)\in \R^4: t\geq w^2, z=1,r\geq b, w=r \right\}\\
Z_=&\defeq\left\{(z,w,r,t)\in \R^4: t\geq w^2, z=0,|r|\leq b, w=0 \right\}.
\end{align*}
Using standard disjunctive programming reformulations, we find that $(\bar z,\bar w,\bar r,\bar t)\in \text{cl conv}(Z_{\text{HA}})$ if and only if there exists additional variables $(z_-,w_-,r_-,t_-,\alpha_-)$, $(z_+,w_+,r_+,t_+,\alpha_+)$ and $(z_=,w_=,r_=,t_=,\alpha_=)$ such that the system 
\begin{align*}
&\alpha_-+\alpha_++\alpha_==1,\; \alpha_-\geq 0,\; \alpha_+\geq 0,\; \alpha_=\geq 0\\
&\bar z=z_-+z_++z_=,\; \bar w=w_-+w_++w_=,\; \bar r=r_-+r_++r_=,\; \bar t=t_-+t_++t_=\\
&t_-\alpha_-\geq w_-^2,\; z_-=\alpha_-,\; r_-\leq -b\alpha_-,\; w_-=r_-\\
&t_+\alpha_+\geq w_+^2,\; z_+=\alpha_+,\; r_+\geq b\alpha_+,\; w_+=r_+\\
&t_=\alpha_=\geq w_=^2,\; z_==0,\; |r_=|\leq b\alpha_=,\; w_==0.
\end{align*}
To obtain the result, we project out the additional variables. First, we remove auxiliary variables $(\bm{z},\bm{w})$ by using the equality constraints, resulting in the simplified system
\begin{align*}
&\alpha_-+\alpha_++\alpha_==1,\; \alpha_-\geq 0,\; \alpha_+\geq 0,\; \alpha_=\geq 0\\
&\bar z=\alpha_-+\alpha_+,\; \bar w=r_-+r_+,\; \bar r=r_-+r_++r_=,\; \bar t=t_-+t_++t_=\\
&t_-\alpha_-\geq r_-^2,\; r_-\leq -b\alpha_-\\
&t_+\alpha_+\geq r_+^2,\; r_+\geq b\alpha_+\\
&t_=\geq 0,\; |r_=|\leq b\alpha_=.
\end{align*}
Next we can project out auxiliary variables $\bm{t}$, leading inequality $\bar t\geq r_-^2/\alpha_-+r_+^2/\alpha_+$, and auxiliary variable $\alpha_=$ using the first equality constraint. The simplified system reads
\begin{align*}
&\alpha_-\geq 0,\; \alpha_+\geq 0,\; \alpha_-+\alpha_+\leq 1\\
&\bar z=\alpha_-+\alpha_+,\; \bar w=r_-+r_+,\; \bar r=r_-+r_++r_=\\
&\bar t\geq r_-^2/\alpha_-+r_+^2/\alpha_+\\
&r_-\leq -b\alpha_-,\; r_+\geq b\alpha_+,\; |r_=|\leq b(1-\alpha_--\alpha_+).
\end{align*}
Finally, we project out $r_=$ since $\bar r=r_-+r_++r_=\Leftrightarrow \bar r=\bar w+r_=$, and thus $r_=$ satisfying constraints exists if and only $|\bar r-\bar w|\leq b(1-\alpha_--\alpha_+)$. The result of Proposition~\ref{prop:hull} then follows by renaming variables as $\alpha_-\leftrightarrow z_-$, $\alpha_+\leftrightarrow z_+$, $r_-\leftrightarrow -w^-$ and $r_+\leftrightarrow w^+$.

\subsection{Proof of Proposition \ref{prop:perspectiveLoss} and Proposition \ref{prop:perspectiveLossSimple} }\label{app:h_closed_form}

We first establish a key lemma about the structure of optimal solutions of the relaxation problem.
\begin{lemma}\label{lem:complementarity}
Any optimal solution of \eqref{eq:LTS_perspHA_relax} with any given $\bm{\ell^-},\, \bm{\ell^+},\, \bm{u^-}, \bm{u^+}$ satisfies $w_i^+ w_i^- = 0$ and $z_i^+ z_i^- = 0$ for all $i \in [n]$.
\end{lemma}

\proof{Proof.} Given an index $i$, we discuss the following three cases.

\noindent\textbf{Case 1: $u^-_i=u^+_i=0$.} This implies $z_i^+ = z_i^- = 0$. The perspective terms $\frac{d_i(w_i^+)^2}{z_i^+}$ and $\frac{d_i(w_i^-)^2}{z_i^-}$ then force $w_i^+ = w_i^- = 0$. Thus $w_i^+ w_i^- = z_i^+ z_i^- = 0$.

\noindent\textbf{Case 2: $\ell^-_i=1$ or $\ell^+_i=1$.} This implies $z_i^+ + z_i^- = 1$. Constraint \eqref{eq:LTS_perspHA_abs} becomes $|r_i - w_i| \leq 0$, which forces $w_i = r_i$. For a fixed residual $r_i$, we determine the optimal values of the auxiliary variables.

If $\ell^+_i=1$ (i.e., $w_i \geq 0$), then $r_i = w_i \geq 0$. If the problem is feasible, the objective is minimized by setting $z_i^+ = 1$, $z_i^- = 0$, $w_i^+ = r_i$, and $w_i^- = 0$, which satisfies $w_i = w_i^+ - w_i^- = r_i$ and minimizes the perspective term.

If $\ell^-_i=1$ (i.e., $w_i \leq 0$), then $r_i = w_i \leq 0$. If the problem is feasible, the objective is minimized by setting $z_i^+ = 0$, $z_i^- = 1$, $w_i^+ = 0$, and $w_i^- = -r_i$, which satisfies $w_i = w_i^+ - w_i^- = r_i$ and minimizes the perspective term.

In both cases, $w_i^+ w_i^- = z_i^+ z_i^- = 0$.

\noindent\textbf{Case 3: $\ell^-_i=\ell^+_i=0$, $u^-_i=u^+_i=1$.} We prove by contradiction. Suppose at an optimal solution, both $z_i^+ > 0$ and $z_i^- > 0$ for some $i$. By constraint \eqref{eq:LTS_perspHA_linear}, we have $w_i^+ \geq \sqrt{2\mu} z_i^+ > 0$ and $w_i^- \geq \sqrt{2\mu} z_i^- > 0$.

Consider a perturbation: replace $(z_i^+, z_i^-, w_i^+, w_i^-)$ by $(\alpha z_i^+, \beta z_i^-, \alpha w_i^+, \beta w_i^-)$ for some $0 < \alpha, \beta < 1$ chosen such that
\begin{equation*}
\alpha w_i^+ - \beta w_i^- = w_i^+ - w_i^-.
\end{equation*}
This equation can be rewritten as $w_i^+(1 - \alpha) = w_i^-(1 - \beta)$. Since $w_i^+, w_i^- > 0$, we can choose $\alpha, \beta \in (0,1)$ satisfying this constraint.

We verify that all constraints remain satisfied:
\begin{itemize}
\item Constraint \eqref{eq:LTS_perspHA_linear}: $\alpha w_i^+ \geq \sqrt{2\mu} \alpha z_i^+$ and $\beta w_i^- \geq \sqrt{2\mu} \beta z_i^-$ hold since the original constraints hold.
\item Constraint \eqref{eq:LTS_perspHA_abs}: The left-hand side $|r_i - w_i| = |r_i - (w_i^+ - w_i^-)|$ is unchanged since $\alpha w_i^+ - \beta w_i^- = w_i^+ - w_i^-$. The right-hand side $\sqrt{2\mu}(1 - z_i)$ increases since the new $z_i = \alpha z_i^+ + \beta z_i^- < z_i^+ + z_i^-$. Thus the constraint remains satisfied.
\item The new $z_i = \alpha z_i^+ + \beta z_i^- < z_i^+ + z_i^-$, so the constraint $z_i \in [0,1]$ remains satisfied.
\end{itemize}

The objective changes only in the perspective terms. The new contribution is
\begin{equation*}
\frac{d_i (\alpha w_i^+)^2}{\alpha z_i^+} + \frac{d_i (\beta w_i^-)^2}{\beta z_i^-} = \alpha \frac{d_i (w_i^+)^2}{z_i^+} + \beta \frac{d_i (w_i^-)^2}{z_i^-} < \frac{d_i (w_i^+)^2}{z_i^+} + \frac{d_i (w_i^-)^2}{z_i^-},
\end{equation*}
since $\alpha, \beta < 1$. This contradicts the optimality of the original solution. Therefore, at any optimal solution, $z_i^+ z_i^- = 0$. By constraint \eqref{eq:LTS_perspHA_linear}, this implies $w_i^+ w_i^- = 0$.
\Halmos \endproof

Lemma \ref{lem:complementarity} allows us to project out the auxiliary variables $z_i^+, z_i^-, w_i^+, w_i^-$ by writing
\begin{equation*}
\frac{d_i (w_i^+)^2}{z_i^+} + \frac{d_i (w_i^-)^2}{z_i^-} = \frac{d_i w_i^2}{z_i},
\end{equation*}
where $z_i = z_i^+ + z_i^-$ and $w_i = w_i^+ - w_i^-$. Note that $\ell^-_i=1$ implies $z^-_i=1$, $z^+_i=0$ and therefore $w^+_i=0$, $w\le 0$. Similarly, $\ell^+_i=1$ implies $w\ge 0$.
We can rewrite  \eqref{eq:phiGen} as 
{\small\begin{subequations}\label{eq:phiGen2}
\begin{align}
\phi(r;\ell^-,\ell^+,u^-,u^+,d)=\min_{w \in \mathbb{R}, z \in [0,1]}& \; \frac{1}{2}(w - r)^2 + \mu z + d \left( \frac{w^2}{z} - w^2 \right)\\
\text{s.t.}\,\,\,\,\,&|w| \geq \sqrt{2\mu} z,\,\,\quad |r - w| \leq \sqrt{2\mu} (1 - z) \label{eq:phiGen2-cons}\\
& z\le u^++u^-,\,\,w\ge0 \text{ if }\ell^+=1,\,\,\,w\le0 \text{ if }\ell^-=1.
\end{align}
\end{subequations}}

We now discuss the following four Cases, which gives us the closed form expression of $\phi$ in Proposition \ref{prop:perspectiveLoss} and Proposition \ref{prop:perspectiveLossSimple}.

\paragraph{\textbf{Case 1: $u^-=u^+=0$.}} We have $z = 0$. The constraint $|r - w| \leq \sqrt{2\mu}$ combined with $w = 0$ (forced by the perspective term) yields $|r| \leq \sqrt{2\mu}$, which gives $$\phi(r;0,0,0,0,d)=\frac{1}{2}r^2+\delta_{[-\sqrt{2\mu},\sqrt{2\mu}]}(r).$$

\paragraph{\textbf{Case 2: $l^+=1$.}} we have $z = 1$ and $w \geq 0$. The constraint $|r - w| \leq 0$ forces $w = r$. Combined with $|w| \geq \sqrt{2\mu}$, we obtain $r \geq \sqrt{2\mu}$, , which gives $$\phi(r;0,1,0,1,d)=\mu+\delta_{[\sqrt{2\mu},\infty]}(r).$$

\paragraph{\textbf{Case 3: $l^-=1$.}} We have $z = 1$ and $w \leq 0$. Similarly, $w = r$ and $w \leq -\sqrt{2\mu}$ yield $r \leq -\sqrt{2\mu}$, which gives $$\phi(r;1,0,1,0,d)=\mu+\delta_{[-\infty,-\sqrt{2\mu}]}(r).$$

\paragraph{\textbf{Case 4: $\ell^-=\ell^+=0$, $u^-=u^+=1$.}}
We first ignore the constraints in \eqref{eq:phiGen2-cons} and derive the closed-form solution for the optimization problem
\begin{equation}\label{eq:phi_opt}
\phi(r) = \min_{w \in \mathbb{R}, z \in [0,1]} \; \frac{1}{2}(w - r)^2 + \mu z + d \left( \frac{w^2}{z} - w^2 \right),
\end{equation}
where $\mu > 0$ and $0 < d < 1/2$. 

\noindent\textbf{Step 1: Minimization over $w$ for fixed $z \in (0,1]$.} 
For any $z > 0$, the objective is strictly convex in $w$. The coefficient of $w^2$ is
\begin{equation*}
a(z) := \frac{1}{2} + d\left(\frac{1}{z} - 1\right) = \frac{1}{2} - d + \frac{d}{z} > 0.
\end{equation*}
The unique minimizer is
\begin{equation}\label{eq:wstar_z}
w^*(z) = \frac{r}{1 + 2d\left(\frac{1}{z} - 1\right)} = \frac{r}{1 - 2d + \frac{2d}{z}}.
\end{equation}
Substituting back and completing the square, the minimal value over $w$ is
\begin{equation}\label{eq:g_z}
g(z) := \min_w \phi(w; z) = \mu z + \frac{1}{2} r^2 - \frac{r^2}{2\left(1 - 2d + \frac{2d}{z}\right)}.
\end{equation}
When $z = 0$, feasibility forces $w = 0$, yielding $g(0) = \frac{1}{2} r^2$. At $z = 1$, we have $w^*(1) = r$ and $g(1) = \mu$. The original problem thus reduces to
\begin{equation*}
\phi(r) = \min_{z \in [0,1]} g(z).
\end{equation*}

\noindent\textbf{Step 2: Stationarity in $z$ and the interior solution.}
Define $D(z) := 1 - 2d + \frac{2d}{z} = \frac{(1-2d)z + 2d}{z}$ for $z \in (0,1]$. Differentiating \eqref{eq:g_z} gives
\begin{equation*}
g'(z) = \mu - \frac{r^2 d}{z^2 D(z)^2}.
\end{equation*}
Setting $g'(z) = 0$ and noting that $z^2 D(z)^2 = N(z)^2$ where $N(z) := (1-2d)z + 2d$, we obtain
\begin{equation*}
N(z) = |r| \sqrt{\frac{d}{\mu}}.
\end{equation*}
Solving for $z$ yields the candidate interior minimizer
\begin{equation}\label{eq:zstar_interior}
z^*(r) = \frac{|r| \sqrt{\frac{d}{\mu}} - 2d}{1 - 2d}.
\end{equation}
This is feasible (i.e., $0 \leq z^* \leq 1$) precisely when $2\sqrt{\mu d} \leq |r| \leq \sqrt{\mu/d}$.

\noindent\textbf{Step 3: Evaluation at the optimizer and boundary regimes.}
When $2\sqrt{\mu d} \leq |r| \leq \sqrt{\mu/d}$, we have $z^*(r) \in (0,1)$ and $g'(z^*) = 0$. Using $N(z^*) = |r| \sqrt{d/\mu}$ and $D(z^*) = N(z^*)/z^*$ in \eqref{eq:g_z}:
\begin{equation*}
\phi(r) = g(z^*) = \mu z^* + \frac{1}{2} r^2 - \frac{r^2 z^*}{2 N(z^*)}.
\end{equation*}
Substituting $z^*$ from \eqref{eq:zstar_interior} and simplifying yields
\begin{equation*}
\phi(r) = \frac{-d r^2 + 2\sqrt{\mu d} |r| - 2\mu d}{1 - 2d}.
\end{equation*}
The optimal $w^*$ in this regime is obtained from \eqref{eq:wstar_z}:
\begin{equation*}
w^*(r) = \frac{r}{D(z^*)} = \frac{r z^*}{N(z^*)} = \text{sign}(r) \sqrt{\frac{\mu}{d}} z^*(r).
\end{equation*}

When $|r| < 2\sqrt{\mu d}$, the stationary $z^*$ in \eqref{eq:zstar_interior} is negative and infeasible. Since $g$ is convex on $[0,1]$, the minimum occurs at $z = 0$, which forces $w = 0$. Thus $\phi(r) = \frac{1}{2} r^2$.

When $|r| > \sqrt{\mu/d}$, we have $z^* > 1$, which is infeasible. The minimum occurs at $z = 1$, where $w^*(1) = r$ and $\phi(r) = \mu$.

\noindent\textbf{Summary.}
Combining the three cases:
\begin{equation*}
\phi(r) = \begin{cases}
\dfrac{1}{2} r^2, & |r| \leq 2\sqrt{\mu d}, \\[8pt]
\dfrac{-d r^2 + 2\sqrt{\mu d} |r| - 2\mu d}{1 - 2d}, & 2\sqrt{\mu d} \leq |r| \leq \sqrt{\dfrac{\mu}{d}}, \\[8pt]
\mu, & |r| \geq \sqrt{\dfrac{\mu}{d}},
\end{cases}
\end{equation*}
with optimal $(w^*, z^*)$ given by
\begin{equation*}
z^*(r) = \begin{cases}
0, & |r| \leq 2\sqrt{\mu d}, \\[3pt]
\dfrac{|r| \sqrt{\frac{d}{\mu}} - 2d}{1 - 2d}, & 2\sqrt{\mu d} \leq |r| \leq \sqrt{\dfrac{\mu}{d}}, \\[10pt]
1, & |r| \geq \sqrt{\dfrac{\mu}{d}},
\end{cases}
\end{equation*}
\begin{equation*}
w^*(r) = \begin{cases}
0, & |r| \leq 2\sqrt{\mu d}, \\[3pt]
\text{sign}(r) \sqrt{\dfrac{\mu}{d}} z^*(r), & 2\sqrt{\mu d} \leq |r| \leq \sqrt{\dfrac{\mu}{d}}, \\[6pt]
r, & |r| \geq \sqrt{\dfrac{\mu}{d}}.
\end{cases}
\end{equation*}

\noindent\textbf{Derivative of $\phi(r)$.}
Differentiating each piece yields:
\begin{equation*}
\phi'(r) = \begin{cases}
r, & |r| < 2\sqrt{\mu d}, \\[6pt]
\dfrac{-2d r + 2\sqrt{\mu d} \, \text{sign}(r)}{1 - 2d}, & 2\sqrt{\mu d} < |r| < \sqrt{\dfrac{\mu}{d}}, \\[8pt]
0, & |r| > \sqrt{\dfrac{\mu}{d}}.
\end{cases}
\end{equation*}
At the transition points, continuity can be verified directly. At $|r| = 2\sqrt{\mu d}$: the left derivative is $\text{sign}(r) \cdot 2\sqrt{\mu d}$, and the right derivative is $\frac{-2d \cdot 2\sqrt{\mu d} \, \text{sign}(r) + 2\sqrt{\mu d} \, \text{sign}(r)}{1-2d} = \text{sign}(r) \cdot 2\sqrt{\mu d}$. At $|r| = \sqrt{\mu/d}$: the left derivative is $\frac{-2d \sqrt{\mu/d} \, \text{sign}(r) + 2\sqrt{\mu d} \, \text{sign}(r)}{1-2d} = 0$. Thus $\phi'(r)$ is continuous everywhere.

\noindent\textbf{Checking feasibility in \eqref{eq:phiGen2-cons}.}
\begin{itemize}
\item When $|r| \leq 2\sqrt{\mu d}$: We have $z^* = w^* = 0$. Both constraints reduce to $0 \geq 0$ and $|r| \leq \sqrt{2\mu}$. The latter holds since $|r| \leq 2\sqrt{\mu d} \leq \sqrt{2\mu}$ (using $d \leq 1/2$).
\item When $2\sqrt{\mu d} \leq |r| \leq \sqrt{\mu/d}$: We have $w^* = \text{sign}(r) \sqrt{\mu/d} z^*$. The first constraint becomes $\sqrt{\mu/d} z^* \geq \sqrt{2\mu} z^*$, which holds since $d \leq 1/2$. For the second constraint, we compute $|r - w^*| = |r| - |w^*| = |r| - \sqrt{\mu/d} z^*$. Substituting $z^*$ and simplifying shows that this equals $\sqrt{2\mu}(1 - z^*)$.
\item When $|r| \geq \sqrt{\mu/d}$: We have $z^* = 1$ and $w^* = r$. The constraints become $|r| \geq \sqrt{2\mu}$ and $|r - r| = 0 \leq 0$, both of which hold.
\end{itemize}

Since the optimal $(w^*, z^*)$ satisfies constraints \eqref{eq:phiGen2-cons}, they are also the optimal solutions to $\phi(r;0,0,1,1,d)$. Moreover, the constraints \eqref{eq:phiGen_opt} can be removed.

\subsection{Proof of Proposition \ref{thm:dual_bound}}\label{app:dual_bound_proof}

We first establish a preliminary lemma.
\begin{lemma}[Gradient-suboptimality inequality]\label{lem:grad_gap}
If $\Phi: \mathbb{R}^p \to \mathbb{R}$ is $\tilde{\lambda}$-strongly convex and differentiable, then for any $\bm{x} \in \mathbb{R}^p$,
\begin{equation*}
\min_{\bm{u} \in \mathbb{R}^p} \Phi(\bm{u}) \geq \Phi(\bm{x}) - \frac{1}{2\tilde{\lambda}} \|\nabla \Phi(\bm{x})\|_2^2.
\end{equation*}
\end{lemma}
\proof{Proof.}
Strong convexity yields for all $\bm{u}$:
$\Phi(\bm{u}) \geq \Phi(\bm{x}) + \langle \nabla \Phi(\bm{x}), \bm{u} - \bm{x} \rangle + \frac{\tilde{\lambda}}{2} \|\bm{u} - \bm{x}\|_2^2$.
The right side is minimized in $\bm{u}$ at $\bm{u} = \bm{x} - \tilde{\lambda}^{-1} \nabla \Phi(\bm{x})$, giving the result.
\Halmos \endproof
As mentioned in Section~\ref{sec:persp}, the perspective relaxation \eqref{eq:perspProjected} is designed to be $\tilde{\lambda}$-strongly convex. Since the augmented Lagrangian \eqref{eq:Lagnode} is obtained by removing some constraints in \eqref{eq:perspProjected} and adding some piecewise-linear/quadratic functions, it is also $\tilde{\lambda}$-strongly convex. Therefore,
\begin{equation*}
    \bar\zeta(\bm{\bar\nu})=\min_{\bm{\beta}}L_\rho(\bm{\beta;\bm{\bar\nu}}) \ge  L_\rho({\bm{\bar\beta}}, \bm{\bar\nu}) - \frac{\|\nabla_{\bm{\beta}} L_\rho({\bm{\bar\beta}}, \bm{\bar\nu})\|_2^2}{2\tilde{\lambda}}
\end{equation*}
where the right-hand side provides a lower bound of the relaxation problem.
\Halmos \endproof

\subsection{Computing Incumbent Solutions}\label{subsect:incumbent}

We obtain incumbent solutions in the BnB procedure by solving an $\ell_2$-regularized estimation problem with the outlier set $\mathcal{S} \subset [n]$ fixed. Specifically, for a given support $\mathcal{S}$, we solve the original problem \eqref{eq:LTS} with $z_i = 1$ for $i \in \mathcal{S}$ and $z_i = 0$ otherwise:
\begin{equation}\label{eq:upper}
F^*(\mathcal{S}) := \min_{\bm{\beta} \in \mathbb{R}^p} \; \frac{1}{2} \sum_{i \notin \mathcal{S}} (y_i - \bm{x}_i^\top \bm{\beta})^2 + \frac{\lambda}{2} \|\bm{\beta}\|_2^2 + \mu |\mathcal{S}|.
\end{equation}
This is a weighted ridge regression problem with a closed-form solution. Let $\bm{W} = \text{Diag}(1 - \bm{z})$ be the diagonal weight matrix. The optimal coefficient vector is
\begin{equation}\label{eq:upper_beta}
\bm{\beta}^*(\mathcal{S}) = (\bm{X}^\top \bm{W} \bm{X} + \lambda \bm{I})^{-1} \bm{X}^\top \bm{W} \bm{y}.
\end{equation}

Before constructing the BnB tree, following \citep{rousseeuw2006computing}, we compute an initial incumbent solution using a heuristic based on alternating minimization. The method alternates between optimizing over $\bm{\beta}$ for fixed $\bm{z}$ and optimizing over $\bm{z}$ for fixed $\bm{\beta}$. Starting with $\bm{z} = \bm{0}$ (i.e., no outliers), we iterate as follows:
\begin{enumerate}
\item \textbf{Solve for $\bm{\beta}$}: Given the current $\bm{z}$, compute $\bm{\beta}$ using \eqref{eq:upper_beta}.
\item \textbf{Update $\bm{z}$}: Given the current $\bm{\beta}$, compute residuals $r_i = y_i - \bm{x}_i^\top \bm{\beta}$ and update each $z_i$ by comparing the squared loss to the outlier penalty:
\begin{equation}\label{eq:z_update}
z_i = \begin{cases}
1, & \text{if } \frac{1}{2} r_i^2 > \mu, \\
0, & \text{otherwise}.
\end{cases}
\end{equation}
\end{enumerate}
The update rule \eqref{eq:z_update} follows from the optimality condition: for fixed $\bm{\beta}$, observation $i$ should be an outlier if and only if the penalty $\mu$ is less than the squared loss $\frac{1}{2} r_i^2$. This alternating procedure continues until the outlier set $\mathcal{S} = \{i : z_i = 1\}$ stabilizes, which is guaranteed since the objective decreases monotonically. Algorithm \ref{alg:incumbent} summarizes the method.

\begin{algorithm}[!h]
\begin{algorithmic}[1]
\REQUIRE Data $(\bm{X}, \bm{y})$, parameters $\lambda, \mu$
\STATE Initialize $\bm{z} = \bm{0}$
\REPEAT
    \STATE Compute $\bm{\beta} = (\bm{X}^\top \bm{W} \bm{X} + \lambda \bm{I})^{-1} \bm{X}^\top \bm{W} \bm{y}$ where $\bm{W} = \text{Diag}(1 - \bm{z})$
    \STATE Compute residuals $r_i = y_i - \bm{x}_i^\top \bm{\beta}$ for all $i \in [n]$
    \STATE Update $z_i = \mathbf{1}\{\frac{1}{2} r_i^2 > \mu\}$ for all $i \in [n]$
\UNTIL{the outlier set $\mathcal{S} = \{i : z_i = 1\}$ does not change}
\RETURN $\bm{\beta}$, outlier set $\mathcal{S}$
\end{algorithmic}
\caption{Computing initial incumbent via alternating minimization}
\label{alg:incumbent}
\end{algorithm}

At each node in the BnB tree, we further refine the upper bound by solving problem \eqref{eq:upper} with the support set determined by rounding the relaxed integer variables: $\mathcal{S} = \{i : z_i^-+z_i^+ \geq 0.5\}$.
Note that variables $z_i^-+z_i^+$ are only implicitly computed, and this is equivalent to choosing
\begin{equation*}
    \mathcal{S} = \left\{i : |y_i-\bm{x_i^\top\beta}| \geq  \Bigl(\tfrac12 + d\Bigr)\sqrt{\frac{\mu}{d}}\right\} \,\,\,\text{ or }\,\,\, \mathcal{S} = \left\{i : \phi(y_i-\bm{x_i^\top\beta}) \geq  \frac{\mu\left(\tfrac34-d-d^2\right)}{1-2d}\right\}.
\end{equation*}